\documentclass[11pt,a4paper]{article}
\usepackage{amssymb,amsmath,amsfonts,amscd,amsthm,amsxtra,latexsym}
\usepackage[utf8]{inputenc} 
\usepackage[T1]{fontenc}    
\usepackage{stmaryrd}
\usepackage{fullpage}
\theoremstyle{plain}
\newtheorem{theorem}{Theorem}
\newtheorem{proposition}{Proposition}
\newtheorem{lemma}{Lemma}
\newtheorem{corollary}{Corollary}

\theoremstyle{remark}
\newtheorem{definition}{Definition}
\newtheorem{remark}{Remark}
\newtheorem{example}{Example}

\newtheoremstyle{erw}							
  {3.25ex plus1ex minus.2ex}						
  {}									
  {}									
  {}									
  {\bfseries\sffamily}							
  {}									
{\newline}								
  {\thmnumber{#1}\thmnote{ #2'.}\vspace{1.5ex plus.2ex}}		

\theoremstyle{erw}
\newtheorem{singletheorem}{Theorem}

\,
\newcommand{\SO}{\mathrm{SO}}

\newcommand{\so}{\mathfrak{so}}

\newcommand{\Sym}{\mathrm{Sym}}

\newcommand{\g}[2]{\langle #1,#2\rangle}
\renewcommand{\gg}{\g{\,\cdot\,}{\,\cdot\,}}
\newcommand{\R}{\mathrm{I\!R}}
\newcommand{\C}{\mathbb{C}}

\newcommand{\dt}{\mathrm{dt}}

\newcommand{\Id}{\mathrm{Id}}

\newcommand{\End}{\mathrm{End}}

\newcommand{\trace}{\mathrm{tr}}

\,

\newcommand{\Hom}{\mathrm{Hom}}

\newcommand{\rmS}{\mathrm{S}}

\newcommand{\scrC}{\mathcal{C}}

\newcommand{\scrN}{\mathcal{N}}
\newcommand{\scrR}{\mathcal{R}}

\newcommand{\scrS}{\mathcal{S}}

\renewcommand{\sl}{\mathfrak{sl}}
\newcommand{\ric}{\mathrm{ric}}
\newcommand{\Ric}{\mathrm{Ric}}

\renewcommand{\d}{\mathrm{d}}

\newcommand{\scal}{\mathrm{s}}

\title{The two-jet of the curvature tensor of an Einstein manifold}
\author{Tillmann Jentsch}
\date{}

\begin{document}\sloppy

\maketitle
\begin{abstract}
The two-jet $(R_p\nabla R_p,\nabla^{2)} R_p)$ of the curvature tensor at some
point $p$ of a pseudo-Riemannian manifold is called Einstein if the Ricci tensor is a multiple of the metric tensor $g_p$
and  additionally its first two covariant derivatives vanish at $p$\,. Following the
Jet Isomorphism Theorem of pseudo-Riemannian geometry, we derive necessary
and sufficient conditions for the Einstein property in terms of the symmetrization of  $(R_p,\nabla R_p,\nabla^{2)} R_p)$ (i.e. in terms of the
Jacobi operator and its first two covariant derivatives along arbitrary geodesics emanating from $p$).  A central role is played by the Weitzenböck
formula for the Laplacian $\d^\nabla \delta^\nabla + \delta^\nabla \d^\nabla$
acting on sections of the vector bundle of algebraic curvature tensors. As an application, we study linear Jacobi relations of
order two on Einstein manifolds.
\end{abstract}

\section{Introduction}
\label{se:introduction}
Let $(M,\gg)$ be a pseudo-Riemannian manifold and $p\in M$\,. We denote by $\nabla$ and $R$ the Levi-Civita connection and the curvature tensor. 
It was shown in~\cite{NO} that the condition $\nabla^{k)} R = 0$ already implies that $M$ is locally symmetric.
In order to find less restrictive conditions on the curvature tensor, recall that the Jacobi operator along some geodesic $\gamma$ is defined by $\scrR_\gamma(t;x,y) := R(x,\dot\gamma(t),\dot\gamma(t),y)$ for all $t\in \R$ and $x,y\in T_{\gamma(t)}M$\,. 
Further, let $\Sym^kV$ denote the  $k$-th symmetric power of an arbitrary vector space $V$\,. Thus the Jacobi operator is a section of $\Sym^2 TM^*$ along $\gamma$\,. 

Furthermore, let $\nabla^{k)}R$ denote the $k$-th covariant derivative of the curvature tensor. Then the $k$-th covariant derivative of $\scrR_{\gamma}$ is given by
\begin{equation}\label{eq:k-ter_Jacobi_on_a_geodesic}
\scrR^{k)}_\gamma(x,y) := \frac{\nabla^{k}}{\dt^k} \scrR_\gamma(x,y) = \nabla^{k)}_{\dot \gamma,\cdots,\dot \gamma}R(\dot \gamma,x ,y ,\dot \gamma)
\end{equation}
which is a section of $\Sym^2 TM^*$ along $\gamma$ again.
However, because of~\eqref{eq:k-ter_Jacobi_on_a_geodesic}, we can define the
$k$-th covariant derivative of the Jacobi operator without reference to a special geodesic as follows:

\bigskip
\begin{definition}
Let $\scrR^{k)}$ denote the section of $\Sym^{k+2}TM^*\otimes \Sym^2 TM^*$ which
is uniquely characterized  by
\begin{equation}\label{eq:k-ter_Jacobi}
\scrR^{k)}(\underbrace{\xi,\cdots,\xi}_{k+2};x,y) := \nabla^{k)}_{\xi,\cdots,\xi} R(x,\xi,\xi,y)
\end{equation}
 for all $p\in M$ and $\xi,x,y\in T_pM$ via the polarization formula. We will call $\scrR^{k)}$ the
symmetrized $k$-th covariant derivative of the curvature tensor. \end{definition}

\subsection{The Einstein condition on the symmetrized two-jet of the  curvature tensor}
\label{se:Einstein_condition}
Since we are interested in local properties of the metric, from now on we assume that $M$ is a
vector space $V$ and $p$ is the origin. Thus $\gg := g|_0$ equips $V$ with the structure of a pseudo-euclidean vector
space. The {\em $k$-jet} associated with $g$ is the collection $(R|_0,\nabla R|_0,\cdots,\nabla^{k)}R|_0)$ obtained by evaluating the
curvature tensor and its first $k$ covariant derivatives at the origin. 
A $k$-jet can also be defined in a purely algebraic manner, see Def.~\ref{de:algebraic_two-jet}. 
In the following an (algebraic) $k$-jet will be denoted by $(R,\nabla
R,\cdots,\nabla^{k)}R)$ (the symbol $|_0$ for evaluation at the origin gets dropped
only in order to keep the notation more easy.)

\bigskip
\begin{definition}
We say that an algebraic $k$-jet $(R,\nabla R,\cdots,\nabla^{k)}R)$  is Einstein if the
Ricci tensor is a multiple of $\gg$ and the vanishing conditions $\nabla\ric =
0,\cdots,\nabla^{k)}\ric = 0$ together hold.
\end{definition}

It is known that $\nabla\ric = 0$ implies that $\nabla R$ is totally trace-free, see Lemma~\ref{le:hirachy}.

Let $\alpha\in \Sym^k V^*$ and $\beta\in \Sym^\ell V^*$\,. 
We define the symmetric product $\alpha \odot \beta \in \Sym^{k + \ell} V^*$ 
as the product of the corresponding polynomial functions, i.e.
\begin{equation}\label{eq:symmetric_product}
\alpha \odot \beta (\xi,\cdots,\xi) := \alpha(\xi)\beta(\xi)\;.
\end{equation}

\bigskip
\begin{theorem}\label{th:main}
An algebraic two-jet $(R,\nabla R,\nabla^{2)} R)$ is Einstein if and only if
$\ric \in \R\,\gg$\,, $\nabla \ric = 0$ and
\begin{equation}\label{eq:traceless_part_alternativ}
\scrR^{2)} - \frac{1}{n + 4}\scrR*\scrR \odot \gg\;.
\end{equation} 
is totally trace-free. Here $\scrR*\scrR$ denotes the algebraic Jacobi operator
associated with the algebraic curvature tensor $R*R$ (the latter will be defined
in~\eqref{eq:def_of_R*R} below.) 
\end{theorem}

\subsection{Linear  Jacobi relations}
\bigskip
\begin{definition}\label{de:linear_Jacobi_relation}
A linear Jacobi relation of order $k$ is a dependency relation
\begin{equation}\label{eq:linear_Jacobi_relation}
\scrR^{k + 1)} = c_{k - 1}\, \gg\, \odot\, \scrR^{k - 1)}\, +  c_{k - 3}\, \gg\, \odot\, \gg\,\odot\, \scrR^{k - 3)} + \cdots
 \end{equation}
in the space of sections of the vector bundle $\Sym^{k + 3}TM^*\otimes\Sym^2TM^*$\,.
\end{definition}

In other words, this means that there exist constants $c_{k-1},c_{k-3},\cdots$ such that 
\begin{equation}\label{eq:linear_Jacobi_relation_for_geodesics}
\scrR^{k+1)}_{\gamma} = c_{k-1}\,\g{\dot \gamma}{\dot \gamma}\, \scrR^{k-1)}_\gamma + c_{k-3}\,\g{\dot \gamma}{\dot \gamma}^2\, \scrR^{k-3)}_\gamma + \cdots\;.
\end{equation}
for all geodesics $\gamma$\,. 


For example, a linear Jacobi relation of order zero means that $\nabla R = 0$\,, i.e. $M$ is locally
symmetric. Further, examples with (minimal) linear Jacobi relations of order  two and four are known, see~\cite{A,J2}. Among them there are certain homogeneous Einstein spaces. 
On the other hand it is an open question whether there exist pseudo-Riemannian spaces with a (minimal) linear Jacobi relation of uneven order at all.
Using Theorem~\ref{th:main} in combination with~\eqref{eq:Weitzenböck_spezial} below, we will prove the following result on linear Jacobi relations of order one:

\bigskip
\begin{corollary}\label{co:main}
Suppose for a pseudo-Riemannian Einstein manifold there exists a constant $c$ such that
\begin{equation}\label{eq:Relation_of_order_two}
\scrR^{2)}_\gamma = c\, \g{\dot \gamma}{\dot \gamma}\, \scrR^{0)}_\gamma
\end{equation}
for all geodesics $\gamma$\,. Then the curvature tensor of $M$ satisfies
\begin{equation}\label{eq:main}
\nabla^*\nabla R = - \frac{(n+4)c}{2} R\;.
\end{equation}
In particular, the constant $c$ is negative on a compact Riemannian manifold.
\end{corollary}

\section{Symmetrized/anti-symmetrized covariant derivatives of the curvature tensor}
\label{se:reformulation}
Instead of working directly with the symmetrization of the covariant derivatives of the curvature tensor defined
in~\eqref{eq:k-ter_Jacobi}, there is no flaw in passing to the
symmetrized/anti-symmetrized covariant derivative, since this contains exactly
the same information. In fact, this has even some advantages for the calculations. The construction is as follows:

Consider the Young diagram of shape $(k+2,2)$ given by 
\begin{equation}\label{eq:Young_symmetrizer}
\begin{array}{|c|c|c|c|c|}
\cline{1-5}
1 & 3 & 5 & \cdots & k+4\\
\cline{1-5}
2 & 4 & \multicolumn{2}{c}{\;\;\;}\\
\cline{1-2}
\end{array}
\end{equation}
for any $k\geq 0$ \,. The corresponding Young symmetrizer defines an endomorphism on covariant
tensors of degree $k + 4$ which can roughly be described as a specific
symmetrization/anti-symmetrization operation (see~\cite[Ch.~6]{FH}). The {\em Schur functor} associated with~\eqref{eq:Young_symmetrizer} assigns to each
vector space $V$ the image $\scrC_k(V)$ of this endomorphism in $\bigotimes^{k+4}V^*$\,. The complexification $\scrC_k(V)\otimes\C$ is known to be an irreducible representation of $\sl(V,\C)$ of highest weight $(k+2) L_1 + 2 L_2$ in the notation from~\cite{FH}, called Weyls construction. Once the vector space $V$ is kept fix, we write $\scrC_k$ instead of
$\scrC_k(V)$\,. For a proof of the following Lemma see for example~\cite[Sec.~3]{FKWC}:

\bigskip
\begin{lemma}\label{le:scrC}
The linear spaces $\scrC_k\subset \bigotimes^{k+4} V^*$ can also be described as follows:
\begin{align}\label{eq:first_Bianchi_identity}
& \scrC_0 := \{R\in \Sym^2(\Lambda^2 V^*)|R\ \text{satisfies the first Bianchi-identity}\}\;,\\
\label{eq:second_Bianchi_identity}
& \scrC_1 := \{\nabla R\in V^*\otimes\scrC_0|\nabla R\ \text{satisfies the second Bianchi identity}\}\;,\\
\label{eq:trivial_Ricci_identity}
&\scrC_k := \Sym^k V^*\otimes \scrC_0 \cap \Sym^{k-1} V^*\otimes \scrC_1\
\text{for $k\geq 2$}\;.
\end{align}
\end{lemma}
Thus $\scrC_0$ is and $\scrC_1$ are the spaces of (algebraic) curvature tensors and their covariant derivatives, respectively. More generally, $\scrC_k$ is the space of (algebraic) linear $k$-jets of curvature tensors on $V$\,, see also~\cite[Def.~4]{J1}. In fact, according to the Jet Isomorphism Theorem of pseudo-Riemannian geometry (see~\cite{J1}), the prefix ``algebraic'' can actually be omitted: let some pseudo-Riemannian metric 
on the differentiable manifold $V$ be given whose curvature tensor satisfies
$R|_0 =0, \ldots, \nabla^{k-1)}R|_0 = 0$ (where again $|_0$  denotes evaluation at the origin). Then $\nabla^{k)}R|_0\in\scrC_k$ by means of the Ricci identity. Further, any element of $\scrC_k$ is obtained in this way.

The {\em Hook Length Formula}~\cite[4.12]{FH} in
combination with~\cite[Lemma~4.26]{FH}  implies that 
\begin{equation}\label{eq:eigenvalue}
\begin{array}{|c|c|c|c|c|}
\cline{1-5}
1 & 3 & 5 & \cdots & k+4\\
\cline{1-5}
2 & 4 & \multicolumn{2}{c}{\;\;\;}\\
\cline{1-2}
\end{array}\; \nabla^{k)}_{x_4,\cdots,x_{k+4}}R(x_1,x_2,x_3,x_4) = 2(k + 3)(k + 2)k!\; \nabla^{k)}_{x_4,\cdots,x_{k+4}}R(x_1,x_2,x_3,x_4)
\end{equation}
for all $\nabla^{k)}R\in\scrC_k$ and $k\geq 0$\,. For example, the factor $ 2(k + 3)(k + 2)k!$  occurring on the right hand side of~\eqref{eq:eigenvalue} is $12$ for $k=0$\,, $24$ for $k=1$ and $80$ for $k=2$\,, respectively. It is not very difficult to check~\eqref{eq:eigenvalue} from~\eqref{eq:first_Bianchi_identity}-\eqref{eq:trivial_Ricci_identity} via explicit calculations, see~\cite[Sec.~3]{J1}. In fact, in case of $R=0$ the calculations given in Sec.~\ref{se:some_traces} below reduce to the previously mentioned ones.

\bigskip
\begin{example}
Let $\nabla^{k - 2\,\ell)}R\in\scrC_{k -2\,\ell}$ and consider the tensor $\scrR^{k - 2\,\ell)}\odot \gg^{\ell}\in \Sym^{k+2}V^*\otimes\Sym^2$\,, characterized by
\begin{equation}
\scrR^{k - 2\,\ell)}\odot \gg^{\ell}(x,x,x,\cdots,x;y,y) :=  \scrR^{k)}(x;y,y)\g{x}{x}^\ell\ (\text{see also}~\eqref{eq:symmetric_product})\;.
\end{equation}
These are the tensors which occur in~\eqref{eq:linear_Jacobi_relation}.
Further, recall that the Kulkarni-Nomizu product is the linear map 
\[
\owedge:\Sym^{k + 2} V^*\otimes \Sym^2 V^*\to \scrC_k
\]
 given by
\begin{equation}\label{eq:Kulkarni1} \owedge(h_1\otimes h_2)(x_1,\cdots,x_k;x_a,x_b,x_c,x_d) :=  \left\lbrace \begin{array}{l}
\ h_1(x_1,\cdots,x_{k},x_a,x_c)h_2(x_b,x_d)\\ - h_1(x_1,\cdots,x_k,x_b,x_c)h_2(x_a,x_d)\\ 
- h_1(x_1,\cdots,x_{k},x_a,x_d)h_2(x_b,x_c)\\ + h_1(x_1,\cdots,x_k,x_b,x_d)h_2(x_a,x_c)\;.
\end{array}\right . 
\end{equation}
Then, by the very definition of the Young symmetrizer associated with~\eqref{eq:Young_symmetrizer}
\begin{equation}\label{eq:example_for_symmetrized-antisymmetrized}
\begin{array}{c}
\begin{array}{|c|c|c|c|c|}
\cline{1-5}
1 & 3 & 5 & \cdots & k+4\\
\cline{1-5}
2 & 4 & \multicolumn{2}{c}{\;\;\;} \\
\cline{1-2}
\end{array}\; \nabla^{k - 2\,\ell)}R\otimes \gg^\ell(x_1,\cdots,x_{k+4})
= - 2(k+2)!\owedge \scrR^{k-2\,\ell}\odot\gg^\ell(x_1,\cdots,x_{k+4})
\end{array}\;.
\end{equation}\end{example}

\bigskip
\begin{lemma}\label{le:linear_jacobi_relation_alternativ}
A pseudo-Riemannian manifold satisfies a linear Jacobi relation~\eqref{eq:linear_Jacobi_relation} if and only if
\begin{equation}
\begin{array}{l}
\label{eq:linear_Jacobi_relation_alternativ}
\begin{array}{|c|c|c|c|c|}
\cline{1-4}
1 & 2 & \cdots & k+3\\
\cline{1-4}
a & b & \multicolumn{2}{c}{\;\;\;}\\
\cline{1-2}
\end{array}\; \nabla^{(k+1)}_{x_3,\cdots,x_{k+3}}R(x_1,x_a,x_b,x_2)\\
= a_{k - 1}\,\begin{array}{|c|c|c|c|c|}
\cline{1-4}
1 & 2 & \cdots & k+3\\
\cline{1-4}
a & b & \multicolumn{2}{c}{\;\;\;}\\
\cline{1-2}
\end{array}\; \nabla^{(k-1)}_{x_3,\cdots,x_{k+1}}R(x_1,x_a,x_b,x_2)\,\g{x_{k+2}}{x_{k+3}}\\
+ a_{k - 3}\,\begin{array}{|c|c|c|c|c|}
\cline{1-4}
1 & 2 & \cdots & k+3\\
\cline{1-4}
a & b & \multicolumn{2}{c}{\;\;\;}\\
\cline{1-2}
\end{array}\; \nabla^{(k-3)}_{x_3,\cdots,x_{k-1}}R(x_1,x_a,x_b,x_2)\,\g{x_{k}}{x_{k+1}}\g{x_{k+2}}{x_{k+3}}  + \cdots \;.
\end{array}
\end{equation}
\end{lemma}
\begin{proof}
Set
\begin{equation}\label{eq:formal_Jacobi_operator}
\scrN_k := \{h\in\Sym^kV^*\otimes\Sym^2V^*\mid\forall\xi \in V:h(\underbrace{\xi,\cdots,\xi}_{k+1},\,\cdot\, ) = 0\}
\end{equation} Because of the first Bianchi identity, $\scrR^{k - 2\,\ell)}\odot\gg^\ell \in\scrN_{k + 2}$\,. 
Furthermore, the Kulkarni-Nomizu product $\owedge:\scrN_{k+2}\to \scrC_k$ is a non-zero multiple of the  linearization of some isomorphism of
affine jet-spaces appearing in the Jet Isomorphism Theorem of pseudo-Riemannian geometry, see~\cite[Eq.~(48)]{J1}. In particular, the Kulkarni-Nomizu product has a trivial kernel on $\scrN_{k+2}$ for any $k\geq 0$\,.
The assertion follows from~\eqref{eq:example_for_symmetrized-antisymmetrized}.
\end{proof}

\section{Some identities for algebraic curvature tensors}
For simplicity, we restrict our considerations to the Riemannian case. The
generalization to the pseudo-Riemannian case is straightforward. Let $V$ be a Euclidean vector space and $\{e_1,\cdots,e_n\}$ be an orthonormal basis. Given a curvature tensor and some $A\in \bigotimes^k V^*$ we define
\begin{equation}\label{eq:def_of_R*R}
R*A(x_1,\cdots,x_k) := - \sum_{i=1}^k\sum_{j=1}^n R_{x_i,e_j}\cdot A(x_1,\cdots,\overbrace{e_j}^{i},\cdots, x_k)
\end{equation}
Here $R_{x,y}\in \so(V)$ means the curvature endomorphism $R(x,y,\, \cdot\, )$ for all $x,y\in V$\,. Further, the dot $\cdot$ means the usual action of skew-adjoint endomorphisms on arbitrary tensors through algebraic derivation, i.e.
\[
- B\cdot A (x_1,\cdots,x_k) = \sum_{i=1}^k A(x_1,\cdots, B\, x_i,\cdots, x_k)
\] 
for any endomorphism $A$ and $B\in \bigotimes^k V^*$\,. By definition, the tensor $R*A$ has the same symmetries as $A$\,. In particular, if $A$ is a curvature tensor, then  $R*A$ is a curvature tensor again. 

\bigskip
\begin{remark}
A straightforward calculation shows that $R*A$ is the zeroth order term $\Gamma$ which appears
in the definition of the Lichnerowicz Laplacian on covariant tensor fields of valence $k$\,, see~\cite[1.143]{Be}. In~\cite{SW} the operator $R*$ is denoted by $q(R)$\,, see Formula (3.9) there. For example, $R*\alpha (x) = \alpha(\Ric(x))$ for every 1-form $\alpha$\,, e.g. $R* = n\, \Id$ on the vector bundle of 1-forms on the $n$-dimensional round sphere of radius one. 
\end{remark}

\bigskip
\begin{lemma}\label{le:Jacobioperator_of_R*R}
Let two algebraic curvature tensors $R$ and $R'$ be given. Then:
\begin{align}
\label{eq:R*R_alternativ}
& R*R'(x_1,x_2,x_3,x_4) = - \sum_{i=1}^n\left \lbrace 
\begin{array}{l}
2\,R_{x_1,e_i}\cdot R'(e_i,x_2,x_3,x_4) - 2\,R_{x_2,e_i}\cdot R'(e_i,x_1,x_3,x_4) \\
+ R'_{x_2,x_4}\cdot \ric(x_1,x_3) - R'_{x_2,x_3}\cdot \ric(x_1,x_4)\\
+ R'_{x_1,x_3}\cdot \ric(x_2,x_4) - R'_{x_1,x_4}\cdot\ric(x_2,x_3) \\
\end{array}\right .\\
\label{eq:Jacobioperator_of_R*R_alternativ}
& R*R'(x,y,x,y) = - 2\, R'_{x,y}\cdot\ric(x,y) - 4\,\sum_{i=1}^n R_{x,e_i} \cdot R'(e_i,y,x,y)\;,\\
\label{eq:Ricci_tensor_of_R*R}
& \sum_{i=1}^nR*R'(x,e_i,y,e_i) = - R*\ric'(x,y)\;,\\
\label{eq:ScalarCurvature_of_R*R}
& \sum_{i=1}^nR*\ric'(e_i,e_i) = 0\;.
\end{align}
\end{lemma} 
\begin{proof}
For~\eqref{eq:R*R_alternativ}: Since $R'$ has the algebraic properties of a curvature tensor it follows immediately from~\eqref{eq:def_of_R*R} that 
\begin{equation}\label{eq:Jacobioperator_of_R*R}
R*R'(x,y,x,y) = - 2\, \sum_{i=1}^n R_{x,e_i} \cdot R'(e_i,y,x,y) + R_{y,e_i} \cdot R'(e_i,x,y,x)\;.
\end{equation}
Further, let us denote r.h.s. of this equation by $\tilde R(x_1,x_2,x_3,x_4)$\,. 
Thus it is easy to see that $\tilde R(x,y,y,x) = R*R'(x,y,y,x)$\,. Further, a straightforward calculation (using the first Bianchi identity for $R'$) shows that $\tilde R(x_1,x_2,x_3,x_4) = \tilde R(x_3,x_4,x_1,x_2)$ (pair symmetry). Thus~\eqref{eq:Jacobioperator_of_R*R} already implies~\eqref{eq:R*R_alternativ}.

For~\eqref{eq:Jacobioperator_of_R*R_alternativ}: we have
\begin{align}
& R_{x,e_i} \cdot R'(e_i,y,y,x) + R'(\Ric(x),y,y,x) - \g{R(e_i,x,x,)}{R'(e_i,y,y)} =\\
&   - \sum_{i=1}^n\underbrace{R'(e_i,R(x,e_i,y),y,x)}_{=\g{R'(x,y,e_i)}{R(e_i,x,y)}} + \underbrace{R'(e_i,y,R(x,e_i,y),x)}_{=\g{R'(e_i,y,x)}{R(e_i,x,y}} \stackrel{\text{1.~Bianchi}}{=}   \\
& - 2\,\sum_{i=1}^n  \g{R(e_i,x,y)}{R'(e_i,y,x)} + \sum_{i=1}^n \g{R(e_i,x,y)}{R'(e_i,x,y)}\;.
\end{align} 
Further, recall the formula
\[
\sum_{i=1}^n \g{A\,e_i}{B\, e_i} = \trace(B^*\circ A) = \trace(A\circ B^*) = \sum_{i=1}^n \g{A^*\,e_i}{B^*\, e_i}
\]
for any pair of endomorphisms $(A,B)$ on $V$\,. Here $A^*$ means the transpose endomorphism. Applying this formula to $A(u) := R(u,x,y)$ and  $B(u) := R'(u,y,x)$ and using that
$A^*(u) = R(u,y,x)$ and $B^*(u) = R'(u,x,y)$\,, we see that both $\g{R(e_i,x,y)}{R'(e_i,y,x)}$ and $\g{R(e_i,x,y)}{R'(e_i,x,y)}$ are symmetric in $\{x,y\}$\,. 
Hence, 
\begin{equation}\label{eq:warum_R*R_nicht_so_einfach_hinzuschreiben_ist}
\sum_{i=1}^n R_{x,e_i} \cdot R(e_i,y,y,x) - R_{y,e_i} \cdot R(e_i,x,x,y) =
\underbrace{R'(y,x,x,\Ric(y))- R'(x,y,y,\Ric(x))}_{= 2\,R'_{x,y}\cdot \ric(x,y)}\;.
\end{equation}
Therefore, by means of~\eqref{eq:Jacobioperator_of_R*R},
\begin{equation*}
R*R'(x,y,y,x) =  2\, R'_{x,y}\cdot \ric(x,y) - 4\, \sum_{i=1}^n R_{x,e_i} \cdot R'(e_i,y,y,x) 
\end{equation*}
Now~\eqref{eq:Jacobioperator_of_R*R} follows.

For~\eqref{eq:Ricci_tensor_of_R*R}, we can suppose that $x = y$\,.  Then the equation immediately follows from~\eqref{eq:Jacobioperator_of_R*R} and the 
fact that the trace commutes with the action of a skew-symmetric endomorphism. For~\eqref{eq:ScalarCurvature_of_R*R}, we have
\[
\sum_{i,j} R_{e_i,e_j}\cdot \ric(e_j,e_i) = - \sum_{i,j} R_{e_j,e_i}\cdot \ric(e_j,e_i) = - \sum_{i,j} R_{e_i,e_j}\cdot \ric(e_j,e_i)\;.
\]
The result follows.
\end{proof}
Recall that an algebraic curvature tensor $R$ on a Euclidean vector space
$(V,\gg)$ is called Einstein if $\ric = c\,\gg$ for some $c\in \R$. 

\bigskip
\begin{corollary}\label{co:Jacobioperator_of_R*R}
Suppose $R$ is Einstein. Then
\begin{equation}
R*R(x,y,y,x) = -4\,\sum_{i=1}^n R_{x,e_i} \cdot R(e_i,y,y,x) = -4\,\sum_{i=1}^n R_{y,e_i} \cdot R(e_i,x,x,y)\;.
\end{equation}
Further, $R*R$ has vanishing Ricci tensor.
\end{corollary} 

\subsection{The Weitzenb\"ock formula for the Laplacian on the vector bundle of algebraic curvature tensors}
\label{se:Weitzenböck}
In~\cite[Prop.~4.2]{Bou} there was shown a ``Weitzenböck formula''  for the Laplacian
$\d^\nabla\delta^\nabla + \delta^\nabla\d^\nabla$ acting on sections of $\Hom(\Lambda^2 T^*M,\Lambda^2T^*M)$\,, the vector bundle of endomorphisms of the
second exterior power of the cotangent bundle. In
Proposition~\ref{p:Weitzenböck} and Corollary~\ref{co:Weitzenböck_spezial} below we will rewrite this formula
for sections of the subbundle $\scrC_0(TM)\subset \Hom(\Lambda^2 T^*M,\Lambda^2T^*M)$\,, the vector bundle
of algebraic curvature tensors, and specifically for the curvature tensor itself. This will be the cornerstone for our further calculations.

\bigskip
\begin{definition}\label{de:algebraic_two-jet}
Let $(V,\gg)$ be a pseudo Euclidean vector space. 
\begin{enumerate}
\item
A triple $(R',\nabla R',\nabla^{2)}R')\in \scrC_0\oplus (V^*\otimes \scrC_0)\oplus (V^*\otimes V^* \otimes \scrC_0)$ is briefly called an (algebraic) two-jet if
\begin{align*}
& \text{the Ricci-identity:}\ \forall x,y\in V: \nabla^{2)}_{x,y}R - \nabla^{2)}_{y,x}R = R_{x,y} \cdot R\\
& \ \ \ \ \ \ \ \ \ \ \ \ \ \ \ \ \ \ \ \ \ \ \ \ \ \ \ \ \ \ \ \ \ \ \ \text{and} \\
& \text{the second Bianchi identity:}\ \nabla R\in\scrC_1\;, \nabla^{2)}R \in V^*\otimes \scrC_1
\end{align*}
together hold, see~\eqref{eq:second_Bianchi_identity}.
\item More generally, given some $R\in\scrC_0$ (see~\eqref{eq:first_Bianchi_identity}), a triple $(R',\nabla R',\nabla^{2)}R')\in \scrC_0\oplus (V^*\otimes \scrC_0)\oplus (V^*\otimes V^* \otimes \scrC_0)$ is called the (algebraic)
two-jet of a section of the vector bundle of algebraic curvature tensors if the Ricci-identity $\nabla^{2)}_{x,y}R' - \nabla^{2)}_{y,x}R' = R_{x,y} \cdot R'$ holds for all $x,y\in V$\,. 
\end{enumerate}
\end{definition} 
According to the Jet Isomorphism Theorem, a triple  $(R,\nabla R,\nabla^{2)}R)$ is an algebraic two-jet if and only if there exists some
pseudo-Riemannian metric on $V$ with $g_0 = \gg$ and such that $(R,\nabla
R,\nabla^{2)}R)$ is the two-jet of the curvature tensor evaluated at the origin, see Sec.~\ref{se:Einstein_condition}.
Similarly, an algebraic curvature tensor $R$ together with a triple  $(R',\nabla R',\nabla^{2)}R')$ defines the algebraic two-jet of a section of the vector bundle of algebraic curvature tensors
if and only if there exists some pseudo-Riemannian metric with $g_0 = \gg$ and whose curvature tensor at the origin is $R$ such that $(R',\nabla R',\nabla^{2)}R')$ is the two-jet of a section of the vector bundle of algebraic curvature tensors evaluated at the origin.

Further, recall that the divergence of $R$  and the exterior derivative of the Ricci tensor are defined as follows:
\begin{align}\label{eq:divergence_of_R}
& \delta^\nabla_x R(y,z) := - \sum_{i=1}^n \nabla_{e_i} R(e_i,x,y,z)\;,\\
& \d^\nabla \ric(x,y,z) := \nabla_x\ric (y,z) - \nabla_y\ric(x,z)
\end{align}
for all $x,y,z\in V$\,. Then the second Bianchi identity implies
\begin{equation}\label{eq:contracted_second_Bianchi_identity}
\delta^\nabla_z R(x,y) = \d^\nabla \ric(x,y,z)\;.
\end{equation}

Next we state the Weitzenböck formula for the Laplacian $\d^\nabla\delta^\nabla + \delta^\nabla\d^\nabla$ acting on the vector bundle of algebraic curvature tensors over some pseudo-Riemannian manifold, see also Remark~\ref{re:Weitzenböck} below.

\bigskip
\begin{proposition}\label{p:Weitzenböck}
Let an algebraic curvature tensor $R$ and the two-jet $(R',\nabla R',\nabla^{2)}R')$ of a section of the vector bundle of algebraic curvature tensors be given. Then for all $x_1,x_2,x_3,x_4\in V$\,:
\begin{equation}\label{eq:Weitzenböck}
(\d^\nabla\delta^\nabla + \delta^\nabla\d^\nabla) R'(x_1,x_2,x_3,x_4) = \left \lbrace\begin{array}{l}
\nabla^*\nabla R'(x_1,x_2,x_3,x_4) + \frac{1}{2}  R*R'(x_1,x_2,x_3,x_4)\\
+ \frac{1}{2} \left ( \begin{array}{cc} R'_{x_2,x_4}\cdot \ric(x_1,x_3) & - R'_{x_2,x_3}\cdot \ric(x_1,x_4)\\
+ R'_{x_1,x_3}\cdot \ric(x_2,x_4) & -  R'_{x_1,x_4}\cdot\ric(x_2,x_3)
\end{array} \right )
\end{array}\right .\;.
\end{equation}
\end{proposition}
\begin{proof}
It is straightforward that 
\begin{equation*}
(\d^\nabla\delta^\nabla + \delta^\nabla\d^\nabla) R'(x_1,x_2,x_3,x_4) = \nabla^*\nabla R'(x_1,x_2,x_3,x_4) - \sum_{i=1}^n \big ( R_{x_1,e_i}\cdot R'(e_i,x_2,x_3,x_4) - R_{x_2,e_i}\cdot R'(e_i,x_1,x_3,x_4) \big )
\end{equation*}
(cf. the proof of~\cite[Prop.~4.2]{Bou}.) Now~\eqref{eq:Weitzenböck} follows from~\eqref{eq:R*R_alternativ}.
\end{proof}

Note, the term in~\eqref{eq:Weitzenböck} involving the Ricci-curvature vanishes after projection to the vector bundle of algebraic curvature tensors.

\bigskip
\begin{remark}\label{re:Weitzenböck}
For every two-jet $(R',\nabla R',\nabla^{2)} R)$ of a section of the vector bundle of algebraic curvature tensors
\begin{equation}\label{eq:strict_Weitzenböck_formula}
\frac{1}{12}\,\begin{array}{|c|c|}
\cline{1-2}
1 & 3 \\
\cline{1-2}
2 & 4 \\
\cline{1-2}
\end{array}\; (\d^\nabla\delta^\nabla + \delta^\nabla\d^\nabla) R'(x_1,x_2,x_3,x_4) = \nabla^*\nabla R'(x_1,x_2,x_3,x_4) + \frac{1}{2}  R*R'(x_1,x_2,x_3,x_4)\;.
\end{equation}
Here the Laplacian $(\d^\nabla\delta^\nabla + \delta^\nabla\d^\nabla)$ followed by the Young projector is a parallel second order differential operator on the vector bundle of algebraic curvature tensors  $\scrC_0(TM)$  for every pseudo-Riemannian manifold $M$\,, see~\cite[Sec.~3]{SW}. Further, recall the decomposition of a curvature tensor
\begin{equation}\label{eq:decomposition_of_algebraic_curvature_tensors}
R = \frac{\scal}{2n(n-1)} \gg \owedge \gg + \frac{1}{n-2} g\owedge \ric + W
\end{equation}
(where $\owedge$ denotes the Kulkarni-Nomizu product, $\scal$ the scalar curvature, $\ric$ the Ricci tensor and $W$ is the conformal Weyl tensor.) It follows from~\eqref{eq:Ricci_tensor_of_R*R} and~\eqref{eq:ScalarCurvature_of_R*R} that any term occurring in~\eqref{eq:strict_Weitzenböck_formula} respects the corresponding splitting of $\scrC_0(TM)$ into subbundles. Thus~\eqref{eq:strict_Weitzenböck_formula} is a Weitzenböck formula  on any of these vector bundles in the strict sense of~\cite[Sec.~3]{SW}.
\end{remark}

Specifically, we obtain the following formula for the curvature tensor of some pseudo-Riemannian manifold:

\bigskip
\begin{corollary}\label{co:Weitzenböck_spezial}
Let a two-jet $(R,\nabla R,\nabla^{2)}R)$ be given. Then
\begin{equation}\label{eq:Weitzenböck_spezial}
\nabla^*\nabla R(x_1,x_2,x_3,x_4) = \frac{1}{4}\;\begin{array}{|c|c|}
\cline{1-2}
1 & 3 \\
\cline{1-2}
2 & 4 \\
\cline{1-2}
\end{array}\; \nabla^{2)}_{x_1,x_3}\ric(x_2,x_4)  - \frac{1}{2}\,R*R(x_1,x_2,x_3,x_4)\;.
\end{equation}
In particular, if $\nabla^{2)}\ric = 0$ (e.g. the two-jet is Einstein), then
\begin{equation}\label{eq:Weitzenböck_spezial_fuer_Einstein}
\nabla^*\nabla R(x_1,x_2,x_3,x_4) = - \frac{1}{2}\,R*R(x_1,x_2,x_3,x_4)\;.
\end{equation}
\end{corollary}
\begin{proof}
\begin{align*}
&\begin{array}{|c|c|}
\cline{1-2}
1 & 3 \\
\cline{1-2}
2 & 4 \\
\cline{1-2}
\end{array}\; \nabla^{2)}_{x_1,x_3}\, \ric(x_2,x_4)
& = \left \lbrace \begin{array}{ll} 
+ 2\, \nabla^{2)}_{x_1,x_3}\, \ric(x_2,x_4) & + 2\, \nabla^{2)}_{x_3,x_1}\cdot \ric(x_2,x_4)\\
- 2\, \nabla^{2)}_{x_2,x_3}\, \ric(x_1,x_4) & - 2\, \nabla^{2)}_{x_3,x_2}\cdot \ric(x_1,x_4)\\
- 2\, \nabla^{2)}_{x_1,x_4}\, \ric(x_2,x_3) & - 2\, \nabla^{2)}_{x_4,x_1}\cdot \ric(x_2,x_3)\\
+ 2\, \nabla^{2)}_{x_2,x_4}\, \ric(x_1,x_3) & + 2\, \nabla^{2)}_{x_4,x_2}\cdot \ric(x_1,x_3)
\end{array}\right.
\end{align*}
We use the Ricci identity and the second Bianchi identity
via~\eqref{eq:contracted_second_Bianchi_identity} to see that the last expression is
\begin{align*}
&\left . \begin{array}{ll}
+ 4\, \nabla^{2)}_{x_3,x_1}\, \ric(x_2,x_4) & + 2\, R_{x_1,x_3}\cdot \ric(x_2,x_4)\\
- 4\, \nabla^{2)}_{x_3,x_2}\, \ric(x_1,x_4) & - 2\, R_{x_2,x_3}\cdot\ric(x_1,x_4) \\
- 4\, \nabla^{2)}_{x_4,x_1}\, \ric(x_2,x_3) & - 2\, R_{x_1,x_4}\cdot \ric(x_2,x_3)\\
+ 4\, \nabla^{2)}_{x_4,x_2}\, \ric(x_1,x_3) & + 2\, R_{x_2,x_4}\cdot \ric(x_1,x_3)\\
\end{array}\right \rbrace
\stackrel{\eqref{eq:contracted_second_Bianchi_identity}}{=} \left \lbrace \begin{array}{ll} 
 + 4\, \d^\nabla\delta^\nabla R(x_3,x_4,x_1,x_2) & + 4\, \delta^\nabla\underbrace{\d^\nabla R}_{\stackrel{\text{2.\ Bianchi}}= 0}(x_3,x_4,x_1,x_2)\\
+ 2\, R_{x_2,x_4}\cdot \ric(x_1,x_3) & - 2\, R_{x_1,x_4}\cdot \ric(x_2,x_3)\\
+ 2\, R_{x_1,x_3}\cdot \ric(x_2,x_4) & - 2\, R_{x_2,x_3}\cdot\ric(x_1,x_4) \\
\end{array}\right .\\
& \stackrel{\eqref{eq:Weitzenböck}}= 4 \nabla^*\nabla R(x_1,x_2,x_3,x_4) + 2 R*R(x_1,x_2,x_3,x_4)\;.
\end{align*}
The result follows.
\end{proof}

\section{Traces of the associated linear two-jet}
Let $V$ be a Euclidean vector space  and $\{e_1,\ldots,e_n\}$ be an orthonormal basis of $V$\,. Let $A\in \bigotimes^k V^*$ be a covariant $k$-tensor. We put
\begin{equation}
\trace_{i,j} A(x_1,\cdots,\hat x_i,\cdots ,\hat x_j, \cdots x_k) := \sum_{i,j=1}^n A(x_1,\cdots,e_i,\cdots ,e_j, \cdots x_k)\;,
\end{equation}
the trace of $A$ with respect to the variables $x_i$ and $x_j$\,.

Given $\nabla^{2)}R\in \scrC_2$ there exist exactly three
different traces: the second covariant derivative of the Ricci tensor
\begin{equation}\label{eq:Ricci_tensor_of_a_linear_two-jet}
  \nabla^{2)}_{x_1,x_2}\ric(y_1,y_2) =  -\sum_{i=1}^n\nabla^{2)}_{x_1,x_2}R(y_1,e_i,y_2,e_i)\;,\\
\end{equation}
the covariant derivative of the divergence
\begin{equation}\label{eq:divergence_of_the_curvature_tensor_of_a_linear_two-jet}
\nabla_x\delta_{y_1}^\nabla R (y_2,y_3) := - \sum_{i=1}^n\nabla^{2)}_{x,e_i}R(e_i,y_1,y_2,y_3)\;,\\
\end{equation}
and the rough Laplacian given by
\begin{equation}\label{eq:rough_Laplacian_of_a_linear_two-jet}
\nabla^*\nabla R (x_1,x_2,x_3,x_4) := - \sum_{i=1}^n\nabla^{2)}_{e_i,e_i}R(x_1,x_2,x_3,x_4)\;.\;
\end{equation}

Moreover, there are the following relations between the three traces:

\bigskip
\begin{lemma}\label{le:hirachy}
Let  $\nabla^{2)}R\in\scrC_2$ be given. Then
\begin{align}
\label{eq:hirachy_1}
&\nabla_{x_1}\delta^\nabla_{x_2} R(x_3,x_4)  & = &\nabla^{2)}_{x_1,x_3}\ric(x_2,x_4) - \nabla^{2)}_{x_1,x_4}\ric(x_2,x_3)\;,\\
\label{eq:hirachy_2}
&\nabla^*\nabla R(x_1,x_2,x_3,x_4) & = & \frac{1}{4}\, \begin{array}{|c|c|}
\cline{1-2}
1 & 3 \\
\cline{1-2}
2 & 4 \\
\cline{1-2}
\end{array}\;\nabla^{2)}_{x_1,x_3}\ric(x_2,x_4)\\
\label{eq:hirachy_3}
& & = & \nabla_{x_1}\delta^\nabla_{x_2}
R(x_3,x_4)- \nabla_{x_2}\delta^\nabla_{x_1} R(x_3,x_4)\;.
\end{align}
In particular,
\begin{equation}\label{eq:hirachy_4}
\nabla^{2)} \ric = 0 \stackrel{\eqref{eq:hirachy_1}}\Rightarrow \nabla\delta^{\nabla}R =
0 \stackrel{\eqref{eq:hirachy_2}}\Rightarrow \nabla^*\nabla R = 0\;.
\end{equation}
\end{lemma}
\begin{proof}
The first equation uses the second Bianchi identity. Further note that~\eqref{eq:hirachy_2} is a special case
of~\eqref{eq:Weitzenböck_spezial} since $(0,0,\nabla^{2)}R)$ is a two-jet by
definition of $\scrC_2$\,. Then~\eqref{eq:hirachy_3} follows from~\eqref{eq:hirachy_1} since
$\nabla^{2)}\ric\in \Sym^2 V^*\otimes \Sym^2 V^*$\, .
\end{proof}
 
Thus $\nabla^{2)} \ric$ should be seen as the {\em essential trace} of
$\nabla^{2)}R$ because its vanishing already implies that $\nabla^{2)}R$ is totally
trace-free according to~\eqref{eq:hirachy_1}-\eqref{eq:hirachy_4}.

\subsection{The associated second covariant derivative of the Ricci tensor} 
\label{se:some_traces}
 For every algebraic two-jet $(R,\nabla R,\nabla^{2)} R)$ we consider the symmetrized/anti-symmetrized second covariant derivative of the curvature tensor
\begin{equation}\label{eq:associated_linear_two-jet_1}
\tilde \nabla^{2)}_{x_5,x_6}\tilde R (x_1,x_2,x_3,x_4) := \begin{array}{|c|c|c|c|}
\cline{1-4}
1 & 3 & 5 & 6 \\
\cline{1-4}
2 & 4 \\
\cline{1-2}
\end{array}\;\nabla^{2)}_{x_5,x_6}R(x_1,x_2,x_3,x_4)\;.
\end{equation}
Then $\tilde \nabla^{2)}\tilde \in \scrC_2$\,. The associated second covariant derivative of the 
Ricci tensor is thus given by 
\begin{equation}\label{eq:definition_of_the_second_covariant_derivative_of_the_Ricci_tensor_of_the_associated_linear_two-jet}
\tilde \nabla^{2)}_{x_5,x_6}\tilde \ric (x_2,x_4) := -\trace_{1,3}\, \begin{array}{|c|c|c|c|}
\cline{1-4}
1 & 3 & 5 & 6 \\
\cline{1-4}
2 & 4 \\
\cline{1-2}
\end{array}\;\nabla^{2)}_{x_5,x_6}R(x_1,x_2,x_3,x_4)\;.
\end{equation}
In view of~\eqref{eq:eigenvalue} and the remarks after that, the Ricci identity implies that there
exists some expression $f(R)$ quadratic in $R$ such that
\begin{equation}\label{eq:difference_1}
\tilde \nabla^{2)}_{x_5,x_6}\ric (x_2,x_4) - 80\, \nabla^{2)}_{x_5,x_6}\ric
(x_2,x_4) = f(R)\;.
\end{equation}
In fact, the explicit expression for $f(R)$ will be given
in~\eqref{eq:difference_2} below.

\bigskip
\begin{lemma}\label{le:Young1}
For every algebraic curvature tensor $R$ on $V$  and all $x_1,\ldots,x_6\in V$
\begin{align}\label{eq:Young10}
&\trace_{1,3}\;\begin{array}{|c|c|c|}
\cline{1-2}
1 & 3 \\
\cline{1-2}
2 & 4 \\
\cline{1-2}
\end{array}\; R_{x_5,x_3}\cdot R(x_1,x_2,x_4,x_6)  = \trace_{1,3}\; \begin{array}{|c|c|c|}
\cline{1-2}
1 & 3 \\
\cline{1-2}
2 & 4 \\
\cline{1-2}
\end{array}\; R_{x_5,x_1}\cdot R(x_3,x_2,x_4,x_6)\\ & = 3\, \left (\begin{array}{ll}
 \sum_{i=1}^n\; \big (R_{x_5,e_i}\cdot R(e_i,x_2,x_4,x_6) & + R_{x_5,e_i}\cdot R(e_i,x_4,x_2,x_6) \big )\\
 + R_{x_5,x_2}\cdot \ric(x_4,x_6) & + R_{x_5,x_4}\cdot \ric(x_2,x_6)
\end{array}\right )\;.
\end{align}
\end{lemma}
\begin{proof}
Equation~\ref{eq:Young1} is clear from~\eqref{eq:Young_symmetrizer}
and either side of the equation is given by
 \begin{align*}
\sum_{i=1}^n\;&\left \lbrace \begin{array}{llll}
 + R_{x_5,e_i}\cdot R(e_i,x_2,x_4,x_6) & + R_{x_5,e_i}\cdot R(e_i,x_2,x_4,x_6)&  + R_{x_5,e_i}\cdot R(e_i,x_4,x_2,x_6) & + R_{x_5,e_i}\cdot R(e_i,x_4,x_2,x_6)\\
 - R_{x_5,x_2}\cdot R(e_i,e_i,x_4,x_6) & - R_{x_5,e_i}\cdot R(x_2,e_i,x_4,x_6) & - R_{x_5,x_2}\cdot R(e_i,x_4,e_i,x_6) & - R_{x_5,e_i}\cdot R(x_2,x_4,e_i,x_6)\\
 - R_{x_5,e_i}\cdot R(x_4,x_2,e_i,x_6) & - R_{x_5,x_4}\cdot R(e_i,x_2,e_i,x_6) & - R_{x_5,e_i}\cdot R(x_4,e_i,x_2,x_6) & - R_{x_5,x_4}\cdot R(e_i,e_i,x_2,x_6)\\
 + R_{x_5,x_2}\cdot R(x_4,e_i,e_i,x_6) & + R_{x_5,x_4}\cdot R(x_2,e_i,e_i,x_6) & + R_{x_5,x_2}\cdot R(x_4,e_i,e_i,x_6) & + R_{x_5,x_4}\cdot R(x_2,e_i,e_i,x_6)
\end{array}\right .\;.
\end{align*}
from which the result follows.
\end{proof}

\bigskip
\begin{corollary}\label{co:Young1}
Let $R$ be an algebraic curvature tensor.
\begin{equation}\label{eq:Young1}
\trace_{1,3}\;\begin{array}{|c|c|c|}
\cline{1-3}
1 & 3 & 6\\
\cline{1-3}
2 & 4 & \multicolumn{1}{c}{\;\;\;} \\
\cline{1-2}
\end{array}\;R_{x_5,x_6}\cdot R(x_1,x_2,x_3,x_4) = 6\;\left (\begin{array}{l}
-  2\,R_{x_5,x_6}\cdot \ric(x_2,x_4)\\
+ \sum_{i=1}^n  R_{x_5,e_i}\cdot R(e_i,x_2,x_6,x_4)\\ + \sum_{i=1}^nR_{x_5,e_i}\cdot R(e_i,x_4,x_6,x_2)\\
- R_{x_5,x_2}\cdot \ric(x_4,x_6)\\ - R_{x_5,x_4}\cdot \ric(x_2,x_6)
\end{array}\right )\;.
\end{equation}
\end{corollary}
\begin{proof}
\begin{align*}
\trace_{1,3}\;&\begin{array}{|c|c|c|}
\cline{1-3}
1 & 3 & 6\\
\cline{1-3}
2 & 4 & \multicolumn{1}{c}{\;\;\;} \\
\cline{1-2}
\end{array}\; R_{x_5,x_6}\cdot R(x_1,x_2,x_3,x_4)
= & \underbrace{\trace_{1,3}\;\begin{array}{|c|c|}
\cline{1-2}
1 & 3 \\
\cline{1-2}
2 & 4 \\
\cline{1-2}
\end{array}\; R_{x_5,x_6}\cdot  R(x_1,x_2,x_3,x_4)}_{\stackrel{\eqref{eq:eigenvalue}}=- 12\,R_{x_5,x_6}\cdot \ric(x_2,x_4)}\\
+&\; \underbrace{\trace_{1,3}\;\begin{array}{|c|c|}
\cline{1-2}
1 & 3 \\
\cline{1-2}
2 & 4 \\
\cline{1-2}
\end{array}\; R_{x_5,x_1}\cdot  R(x_3,x_2,x_6,x_4)}_{\stackrel{\eqref{eq:Young10}}= \;\left\lbrace\begin{array}{l}
- 3\, \sum_{i=1}^n R_{x_5,e_i}\cdot R(e_i,x_2,x_4,x_6)\\ - 3\, \sum_{i=1}^n R_{x_5,e_i}\cdot R(e_i,x_4,x_2,x_6) \\
- 3\, R_{x_5,x_2}\cdot \ric(x_4,x_6)\\ - 3\, R_{x_5,x_4}\cdot \ric(x_2,x_6)
\end{array}\right .}
&+ \underbrace{\trace_{1,3}\;\begin{array}{|c|c|}
\cline{1-2}
1 & 3 \\
\cline{1-2}
2 & 4 \\
\cline{1-2}
\end{array}\;R_{x_5,x_3}\cdot  R(x_6,x_2,x_1,x_4)}_{\stackrel{\eqref{eq:Young10}}= \left\lbrace\begin{array}{l}
- 3\, \sum_{i=1}^n R_{x_5,e_i}\cdot R(e_i,x_2,x_4,x_6)\\ - 3\, \sum_{i=1}^nR_{x_5,e_i}\cdot R(e_i,x_4,x_2,x_6) \\
- 3\, R_{x_5,x_2}\cdot \ric(x_4,x_6)\\ - 3\, R_{x_5,x_4}\cdot \ric(x_2,x_6)
\end{array}\right .}\;.
\end{align*}
\end{proof}

For the following lemma let $\scrS_I$ denote the Permutation group of some set $I$\,.


\bigskip
\begin{lemma}\label{le:Young2}
Let some algebraic two-jet $(R,\nabla R,\nabla^{2)}R)$ be given.
\begin{equation}\label{eq:Young2}
- \trace_{1,3}\; \begin{array}{|c|c|}
\cline{1-2}
1 & 3 \\
\cline{1-2}
2 & 4 \\
\cline{1-2}
\end{array}\; \nabla^{2)}_{x_1,x_3}\, R(x_5,x_2,x_6,x_4) =\!\!\!\!
\sum_{\begin{array}{c}(\sigma,\tau)\in\\\scrS_{\{2,4\}}\times\scrS_{\{5,6\}}\end{array}} \left \lbrace 
\begin{array}{l} 
\ \nabla^*\nabla R(x_{\tau(5)},x_{\sigma(2)},x_{\tau(6)},x_{\sigma(4)}) \\
+ \nabla^{2)}_{x_{\sigma(2)},x_{\sigma(4)}}\ric(x_{\tau(5)},x_{\tau(6)})\\
- 2\, \nabla_{x_{\sigma(2)}}\delta_{x_{\tau(5)}}R(x_{\sigma(4)},x_{\tau(6)})\\
- \sum_{i=1}^n R_{x_{\sigma(2)},e_i}\cdot R(e_i,x_{\tau(5)},x_{\sigma(4)},x_{\tau(6)})
\end{array}
\right .\;.
\end{equation}
\end{lemma}
\begin{proof}
$\trace_{1,3}\; \begin{array}{|c|c|}
\cline{1-2}
1 & 3 \\
\cline{1-2}
2 & 4 \\
\cline{1-2}
\end{array}\; \nabla^{2)}_{x_1,x_3}\, R(x_5,x_2,x_6,x_4)$
\begin{align*}
& = \sum_{i=1}^n  \left \lbrace
 \begin{array}{llll}
+ \nabla^{2)}_{e_i,e_i}\, R(x_5,x_2,x_6,x_4) & + \nabla^{2)}_{e_i,e_i}\, R(x_5,x_2,x_6,x_4) & + \nabla^{2)}_{e_i,e_i}\, R(x_5,x_4,x_6,x_2) & + \nabla^{2)}_{e_i,e_i}\, R(x_5,x_4,x_6,x_2)\\
- \nabla^{2)}_{x_2,e_i}\, R(x_5,e_i,x_6,x_4) & - \nabla^{2)}_{e_i,x_2}\, R(x_5,e_i,x_6,x_4) &- \nabla^{2)}_{x_2,e_i}\, R(x_5,x_4,x_6,e_i) & - \nabla^{2)}_{e_i,x_2}\, R(x_5,x_4,x_6,e_i)\\
- \nabla^{2)}_{e_i,x_4}\, R(x_5,x_2,x_6,e_i) & - \nabla^{2)}_{x_4,e_i}\, R(x_5,x_2,x_6,e_i) &- \nabla^{2)}_{e_i,x_4}\, R(x_5,e_i,x_6,x_2) & - \nabla^{2)}_{x_4,e_i}\, R(x_5,e_i,x_6,x_2)\\
+ \nabla^{2)}_{x_2,x_4}\, R(x_5,e_i,x_6,e_i) & + \nabla^{2)}_{x_4,x_2}\, R(x_5,e_i,x_6,e_i) &+ \nabla^{2)}_{x_2,x_4}\, R(x_5,e_i,x_6,e_i) & + \nabla^{2)}_{x_4,x_2}\, R(x_5,e_i,x_6,e_i)
\end{array}
 \right .\\
& =\left \lbrace 
\begin{array}{ll} 
 - 2\,\nabla^*\nabla R(x_5,x_2,x_6,x_4) & - 2\,\nabla^*\nabla\, R(x_5,x_4,x_6,x_2) \\
 + 2\,\nabla_{x_2}\delta_{x_5} R(x_4,x_6) & + \sum_{i=1}^n R_{x_2,e_i}\cdot R(e_i,x_5,x_4,x_6)\\
 + 2\,\nabla_{x_2}\delta_{x_6} R(x_4,x_5) & + \sum_{i=1}^n  R_{x_2,e_i}\cdot R(e_i,x_6,x_4,x_5)\\
 + 2\,\nabla_{x_4}\delta_{x_6}R(x_2,x_5)&  + \sum_{i=1}^n R_{x_4,e_i}\cdot R(e_i,x_6,x_2,x_5) \\
 + 2\,\nabla_{x_4}\delta_{x_5}R(x_2,x_6) &  + \sum_{i=1}^n  R_{x_4,e_i}\cdot R(e_i,x_5,x_2,x_6) \\
 - 2\,\nabla^{2)}_{x_2,x_4}\ric(x_5,x_6) & - 2\, \nabla^{2)}_{x_4,x_2}\ric(x_5,x_6)
\end{array}
\right .\;.
\end{align*}
\end{proof}

\bigskip
\begin{proposition}\label{p:Young1}
For an algebraic two-jet $(R,\nabla R,\nabla^{2)} R)$ consider the associated second covariant derivative of the 
Ricci tensor~\eqref{eq:definition_of_the_second_covariant_derivative_of_the_Ricci_tensor_of_the_associated_linear_two-jet}. We have
\begin{equation}\label{eq:Ricci_tensor_of_the-associated_linear_two-jet_1}
\tilde\nabla^{2)}_{x_5,x_6}\tilde \ric(x_2,x_4) = 2\, \sum_{\begin{array}{c}
 (\sigma,\tau)\in\\
\scrS_{\{2,4\}}\times\scrS_{\{5,6\}}\end{array}}\left \lbrace\begin{array}{l}
 - R*R(x_{\tau(5)},x_{\sigma(2)},x_{\tau(6)},x_{\sigma(4)})\\
 + 2\, R_{x_{\tau(5)},x_{\sigma(2)}}\cdot\ric(x_{\tau(6)},x_{\sigma(4)})\\ 
+ 10\, \nabla^{2)}_{x_{\tau(5)},x_{\tau(6)}}\ric(x_{\sigma(2)},x_{\sigma(4)})
\end{array} \right .
\end{equation}
\end{proposition}
\begin{proof}
Let $\rmS_{\{1,3,5,6\}}$\,, $\scrS_{\{1,3,5\}}$\,, $\scrS_{\{1,3,6\}}$ and $\rmS_{\{5,6\}}$  denote the symmetric groups over
$\{1,3,5,6\}$\,, $\{1,3,5\}$\,, $\{1,3,6\}$ and $\{5,6\}$\,, respectively. In particular, $\rmS_{\{5,6\}}$ is a subgroup of $\rmS_{\{1,3,5,6\}}$\,. Let $\scrS_{\{5,6\}}\backslash \rmS_{\{1,3,5,6\}}$ be the set of right cosets. In the free vector space over $\scrS_{\{5,6\}}\backslash \rmS_{\{3,5,6\}}$ 
\begin{align}\label{eq:symmetrizer}
& \sum_{[\pi]\in\scrS_{\{5,6\}\backslash \scrS_{\{1,3,5,6\}}}} [\pi] = \sum_{\pi\in \scrS_{\{1,3,6\}}} [\pi] + \sum_{\pi\in \scrS_{\{1,3,5\}}}[\pi] - \sum_{\pi\in \scrS_{\{1,3\}}} [\pi] + [(1\, 5)(3\, 6)] + [(1\, 6)(3\, 5)]
\end{align}
(see~\cite[Eq.~46 on p.~13]{J1}.) Thus l.h.s. of~\eqref{eq:Ricci_tensor_of_the-associated_linear_two-jet_1} is given by
\begin{align*}
&\underbrace{-\trace_{1,3}\; \begin{array}{|c|c|c|}
\cline{1-3}
1 & 3 & 5\\
\cline{1-3}
2 & 4 & \multicolumn{1}{c}{\;\;\;} \\
\cline{1-2}
\end{array}\; \underbrace{\nabla^{2)}_{x_6,x_5} R(x_1,x_2,x_3,x_4) +
  \nabla^{2)}_{x_5,x_6} R(x_1,x_2,x_3,x_4)}_{= 2\,\nabla^{2)}_{x_6,x_5}   R(x_1,x_2,x_3,x_4) + R_{x_5,x_6}\cdot R(x_1,x_2,x_3,x_4)}}_{\stackrel{\eqref{eq:eigenvalue},\eqref{eq:Young1}}
{=} 48\,\nabla^{2)}_{x_6,x_5}\ric(x_2,x_4) + 12\,R_{x_6,x_5}\cdot \ric(x_2,x_4) + 6\,\left\lbrace\begin{array}{l}
+ \sum_{i=1}^nR_{x_6,e_i}\cdot R(e_i,x_2,x_5,x_4)\\ + \sum_{i=1}^n R_{x_6,e_i}\cdot R(e_i,x_4,x_5,x_2) \\
 + \sum_{i=1}^n R_{x_6,x_2}\cdot \ric(x_4,x_5)\\ + \sum_{i=1}^n R_{x_6,x_4}\cdot \ric(x_2,x_5)
\end{array}\right .}\\
& \underbrace{- \trace_{1,3}\; \begin{array}{|c|c|c|}
\cline{1-3}
1 & 3 & 6\\
\cline{1-3}
2 & 4 & \multicolumn{1}{c}{\;\;\;} \\
\cline{1-2}
\end{array}\; \nabla^{2)}_{x_5,x_6}R(x_1,x_2,x_3,x_4) +
\nabla^{2)}_{x_6,x_5}R(x_1,x_2,x_3,x_4)}_{\stackrel{\eqref{eq:eigenvalue},\eqref{eq:Young1}}
  {=} 48\,\nabla^{2)}_{x_5,x_6}\ric(x_2,x_4) + 12\,R_{x_5,x_6}\cdot \ric(x_2,x_4) + 6\, \left\lbrace\begin{array}{l}
+ \sum_{i=1}^n R_{x_5,e_i}\cdot R(e_i,x_2,x_6,x_4)\\ + \sum_{i=1}^n R_{x_5,e_i}\cdot R(e_i,x_4,x_6,x_2) \\
+ R_{x_5,x_2}\cdot \ric(x_4,x_6)\\ + R_{x_5,x_4}\cdot \ric(x_2,x_6)
\end{array}\right .}\\
& + \underbrace{\trace_{1,3}\; \begin{array}{|c|c|}
\cline{1-2}
1 & 3 \\
\cline{1-2}
2 & 4 \\
\cline{1-2}
\end{array}\; \nabla^{2)}_{x_5,x_6}R(x_1,x_2,x_3,x_4) +
\nabla^{2)}_{x_6,x_5}R(x_1,x_2,x_3,x_4)}_{\stackrel{\eqref{eq:eigenvalue}}{=}- 12(\nabla^{2)}_{x_5,x_6} \ric(x_2,x_4) + \nabla^{2)}_{x_6,x_5} \ric(x_2,x_4))}\\
& \underbrace{- \trace_{1,3}\;\begin{array}{|c|c|}
\cline{1-2}
1 & 3 \\
\cline{1-2}
2 & 4 \\
\cline{1-2}
\end{array}\; \big (\nabla^{2)}_{x_1,x_3}\, R(x_5,x_2,x_6,x_4) + \nabla^{2)}_{x_1,x_3}\, R(x_6,x_2,x_5,x_4)\big )}_{
\stackrel{\eqref{eq:Young2}}{=} 2\, \left ( 
\begin{array}{ll} 
 + 2\,\nabla^*\nabla R(x_5,x_2,x_6,x_4) & + 2\,\nabla^*\nabla\, R(x_6,x_2,x_5,x_4) \\
 - 2\,\nabla_{x_2}\delta_{x_5} R(x_4,x_6) & - \sum_{i=1}^n R_{x_2,e_i}\cdot R(e_i,x_5,x_4,x_6)\\
 - 2\,\nabla_{x_2}\delta_{x_6} R(x_4,x_5) & - \sum_{i=1}^n  R_{x_2,e_i}\cdot R(e_i,x_6,x_4,x_5)\\
 - 2\,\nabla_{x_4}\delta_{x_6}R(x_2,x_5)&  - \sum_{i=1}^n R_{x_4,e_i}\cdot R(e_i,x_6,x_2,x_5) \\
 - 2\,\nabla_{x_4}\delta_{x_5}R(x_2,x_6) &  - \sum_{i=1}^n  R_{x_4,e_i}\cdot R(e_i,x_5,x_2,x_6) \\
 + 2\,\nabla^{2)}_{x_2,x_4}\ric(x_5,x_6) & + 2\, \nabla^{2)}_{x_4,x_2}\ric(x_5,x_6)
\end{array}
\right )}\;.
\end{align*}
Further, recall that
\begin{align}
& \frac{1}{4} \sum_{\begin{array}{c}
(\sigma,\tau)\in\\
\scrS_{\{2,4\}}\times\scrS_{\{5,6\}} \end{array}} R*R(x_{\tau(6)},x_{\sigma(2)},x_{\tau(5)},,x_{\sigma(4)})\\
& \stackrel{\eqref{eq:Jacobioperator_of_R*R}}{=} - \sum_{i=1}^n\left \lbrace \begin{array}{llll} & R_{x_6,e_i}\cdot R(e_i,x_2,x_5,x_4) & + &  R_{x_6,e_i}\cdot R(e_i,x_4,x_5,x_2) \\
+ & R_{x_5,e_i}\cdot R(e_i,x_2,x_6,x_4) & + & R_{x_5,e_i}\cdot R(e_i,x_4,x_6,x_2)
\end{array}  \right . \\
& \stackrel{\eqref{eq:Jacobioperator_of_R*R}}{=} - \sum_{i=1}^n\left \lbrace \begin{array}{llll} & R_{x_2,e_i} \cdot R(e_i,x_5,x_4,x_6) & + & R_{x_2,e_i}\cdot R(e_i,x_6,x_4,x_5) \\
 + & R_{x_4,e_i}\cdot R(e_i,x_6,x_2,x_5)  & + & R_{x_4,e_i}\cdot R(e_i,x_5,x_2,x_6) \end{array} \right .\;.
\end{align}

We conclude that \begin{equation*}
\tilde\nabla^{2)}_{x_5,x_6}\tilde \ric(x_2,x_4) = \sum_{(\sigma,\tau)\in\scrS_{\{2,4\}}\times\scrS_{\{5,6\}}} \left \lbrace\begin{array}{l}
 2\;\nabla^*\nabla R(x_{\tau(5)},x_{\sigma(2)},x_{\tau(6)},x_{\sigma(4)})\\ - R*R(x_{\tau(5)},x_{\sigma(2)},x_{\tau(6)},x_{\sigma(4)})\\
 + 6\, R_{x_{\tau(5)},x_{\sigma(2)}}\cdot\ric(x_{\tau(6)},x_{\sigma(4)})\\  - 4\,\nabla_{x_{\sigma(2)}}\delta^\nabla_{x_{\tau(5)}} R(x_{\sigma(4)},x_{\tau(6)})\\
 + 2\,  \nabla^{2)}_{x_{\sigma(2)},x_{\sigma(4)}}\ric(x_{\tau(5)},x_{\tau(6)})\\ + 18\, \nabla^{2)}_{x_{\tau(5)},x_{\tau(6)}}\ric(x_{\sigma(2)},x_{\sigma(4)})
 \end{array} \right .\;.
\end{equation*}
Further, 
\begin{align*}
& \sum_{\begin{array}{c}
 (\sigma,\tau)\in\\
\scrS_{\{2,4\}}\times\scrS_{\{5,6\}}\end{array}} \nabla^*\nabla R(x_{\tau(5)},x_{\sigma(2)},x_{\tau(6)},x_{\sigma(4)})\\
& \stackrel{\eqref{eq:Weitzenböck}}{=} \sum_{\begin{array}{c}
 (\sigma,\tau)\in\\
\scrS_{\{2,4\}}\times\scrS_{\{5,6\}}\end{array}}\big (
\underbrace{\d^\nabla\delta^\nabla_{x_{\tau(5)},x_{\sigma(2)}}R(x_{\tau(6)},x_{\sigma(4)})
}_{= \left \lbrace \begin{array}{l}
- \nabla_{x_{\sigma(2)}}\delta^\nabla_{x_{\tau(5)}}R(x_{\tau(6)},x_{\sigma(4)})\\
 + \nabla^2_{x_{\tau(5)},x_{\tau(6)}}\ric(x_{\sigma(2)},x_{\sigma(4)})\\
 - \nabla^2_{x_{\tau(5)},x_{\sigma(4)}}\ric(x_{\sigma(2)},x_{\tau(6)})
\end{array}\right .} -  \frac{1}{2} R*R(x_{\tau(5)},x_{\sigma(2)},x_{\tau(6)},x_{\sigma(4)}) \big ) \;.
\end{align*}
We thus see that 
\begin{equation*}
 \begin{array}{lll} \tilde\nabla^{2)}_{x_5,x_6}\tilde \ric(x_2,x_4) & = & 2\, \sum_{(\sigma,\tau)\in\scrS_{\{2,4\}}\times\scrS_{\{5,6\}}} \left \lbrace\begin{array}{l}
 - R*R(x_{\tau(5)},x_{\sigma(2)},x_{\tau(6)},x_{\sigma(4)}) + 3\, R_{x_{\tau(5)},x_{\sigma(2)}}\cdot\ric(x_{\tau(6)},x_{\sigma(4)})\\ 
\underbrace{- \nabla_{x_{\sigma(2)}}\delta^\nabla_{x_{\tau(5)}} R(x_{\sigma(4)},x_{\tau(6)}) +  \nabla^{2)}_{x_{\sigma(2)},x_{\sigma(4)}}\ric(x_{\tau(5)},x_{\tau(6)})}_{=   \nabla^{2)}_{x_{\sigma(2)},x_{\tau(6)}}\ric(x_{\tau(5)},x_{\sigma(4)})}\\ 
- \nabla^2_{x_{\tau(5)},x_{\sigma(4)}}\ric(x_{\sigma(2)},x_{\tau(6)})\\
+ 10\, \nabla^{2)}_{x_{\tau(5)},x_{\tau(6)}}\ric(x_{\sigma(2)},x_{\sigma(4)})
 \end{array} \right .\\
 & = & 2 \sum_{(\sigma,\tau)\in\scrS_{\{2,4\}}\times\scrS_{\{5,6\}}} \left \lbrace\begin{array}{l}
 - R*R(x_{\tau(5)},x_{\sigma(2)},x_{\tau(6)},x_{\sigma(4)}) + 3\, R_{x_{\tau(5)},x_{\sigma(2)}}\cdot\ric(x_{\tau(6)},x_{\sigma(4)})\\ 
+ \underbrace{\nabla^{2)}_{x_{\sigma(2)},x_{\tau(6)}}\ric(x_{\tau(5)},x_{\sigma(4)}) - \nabla^2_{x_{\tau(5)},x_{\sigma(4)}}\ric(x_{\sigma(2)},x_{\tau(6)})}_{R_{x_{\sigma(2)},x_{\tau(6)}}\cdot \ric(x_{\tau(5)},x_{\sigma(4)})}\\
+ 10\, \nabla^{2)}_{x_{\tau(5)},x_{\tau(6)}}\ric(x_{\sigma(2)},x_{\sigma(4)})
\end{array} \right .
\end{array}
\end{equation*}
The result follows from Lemma~\ref{le:Jacobioperator_of_R*R} together with~\eqref{eq:definition_of_the_second_covariant_derivative_of_the_Ricci_tensor_of_the_associated_linear_two-jet}.
\end{proof} 

We thus conclude
from~\eqref{eq:Ricci_tensor_of_the-associated_linear_two-jet_1} in combination
with the Ricci identity that~\eqref{eq:difference_1} is solved by
\begin{equation}
\label{eq:difference_2}
f(R) := - 80\, R_{x_5,x_6}\cdot \ric(x_{2},x_{4}) + 2\, \sum_{\begin{array}{c}
 (\sigma,\tau)\in\\
\scrS_{\{2,4\}}\times\scrS_{\{5,6\}}\end{array}}\left \lbrace\begin{array}{l}
- R*R(x_{\tau(5)},x_{\sigma(2)},x_{\tau(6)},x_{\sigma(4)})\\
 + 2\, R_{x_{\tau(5)},x_{\sigma(2)}}\cdot\ric(x_{\tau(6)},x_{\sigma(4)})
 \end{array} \right .\;.
\end{equation}

Using the special case~\eqref{eq:Weitzenböck_spezial_fuer_Einstein} of the Weitzenb\"ock formula we conclude from Proposition~\ref{p:Young1}:

\bigskip
\begin{corollary}\label{co:Einstein}
Let an algebraic Einstein two-jet $(R,\nabla R,\nabla^{2)} R)$ be given. Then
\begin{equation}
\tilde\nabla^{2)}_{x_5,x_6}\tilde \ric(x_2,x_4) =  - 4\,\big(R*R(x_5,x_2,x_6,x_4) + R*R(x_5,x_4,x_6,x_2)\big )\;.
\end{equation}
\end{corollary}

\subsection{The associated rough Laplacian} 
Let an algebraic two-jet $(R,\nabla R,\nabla^{2)})$ be given. We aim to calculate the rough Laplacian associated
with~\eqref{eq:associated_linear_two-jet_1} according
to~\eqref{eq:rough_Laplacian_of_a_linear_two-jet}, i.e. the algebraic
curvature tensor given by
\begin{equation}
\label{eq:rough_Laplacian_of_the_associated_linear_two-jet_1}
\tilde\nabla^*\tilde\nabla \tilde R(x_1,x_2,x_3,x_4) := - \trace_{5,6}\, \begin{array}{|c|c|c|c|}
\cline{1-4}
1 & 3 & 5 & 6 \\
\cline{1-4}
2 & 4 \\
\cline{1-2}
\end{array}\;\nabla^{2)}_{x_5,x_6}R(x_1,x_2,x_3,x_4)\;.
\end{equation}
Clearly, we will use~\eqref{eq:hirachy_3} to achieve this goal. Again it is a
priori clear 
that there exists some expression $g(R)$ quadratic in $R$ such that
\begin{equation}\label{eq:difference_3}
\tilde \nabla^*\tilde\nabla_{x_5,x_6}\ric (x_2,x_4) - 80\,
\nabla^*\nabla_{x_5,x_6}\ric (x_2,x_4) = g(R)\;.
\end{equation}

\bigskip
\begin{lemma}\label{le:Young3}
\begin{equation}
\label{eq:hilf4}
\begin{array}{|c|c|c|}
\cline{1-2}
1 & 3  \\
\cline{1-2}
2 & 4 \\
\cline{1-2}
\end{array}\; R_{x_1,x_2}\cdot\ric(x_3,x_4) = 0\;.
\end{equation}
\end{lemma}
\begin{proof}
\begin{align*}
&\begin{array}{|c|c|c|}
\cline{1-2}
1 & 3 \\
\cline{1-2}
2 & 4 \\
\cline{1-2}
\end{array}\; R_{x_2,x_1}\cdot\ric(x_3,x_4) = \left \lbrace \begin{array}{l} R_{x_1,x_2}\cdot\ric(x_3,x_4) + R_{x_1,x_4}\cdot\ric(x_3,x_2)\\
+ R_{x_3,x_2}\cdot\ric(x_1,x_4) + R_{x_3,x_4}\cdot\ric(x_1,x_2)\\
- R_{x_2,x_1}\cdot\ric(x_3,x_4) - R_{x_2,x_4}\cdot\ric(x_3,x_1)\\
- R_{x_3,x_1}\cdot\ric(x_2,x_4) - R_{x_3,x_4}\cdot\ric(x_2,x_1)\\
- R_{x_1,x_2}\cdot\ric(x_4,x_3) - R_{x_1,x_3}\cdot\ric(x_4,x_2)\\
- R_{x_4,x_2}\cdot\ric(x_1,x_3) - R_{x_4,x_3}\cdot\ric(x_1,x_2)\\
+ R_{x_2,x_1}\cdot\ric(x_4,x_3) + R_{x_2,x_3}\cdot\ric(x_4,x_1)\\
+ R_{x_4,x_1}\cdot\ric(x_2,x_3) + R_{x_4,x_3}\cdot\ric(x_2,x_1)
\end{array}\right \rbrace = 0\;.
\end{align*}
\end{proof}

\bigskip
\begin{proposition}\label{p:rough_Laplacian_of_the_associated_two-jet}
Let an algebraic two-jet $(R,\nabla R,\nabla^{2)} R)$ be
given. Consider the rough Laplacian~\eqref{eq:rough_Laplacian_of_a_linear_two-jet}
associated with the linear two-jet~\eqref{eq:associated_linear_two-jet_1}. We have
\begin{equation}\label{eq:Laplace_der_symmetrisierten_zweiten_kovarianten_Ableitung}
\tilde \nabla^*\tilde \nabla \tilde R(x_1,x_2,x_3,x_4) = 80\, \nabla^*\nabla R(x_1,x_2,x_3,x_4) + 16\, R*R(x_1,x_2,x_3,x_4)\;.
\end{equation}
\end{proposition}
\begin{proof}
According to the last heorem,~\eqref{eq:Weitzenböck_spezial} and Lemma~\ref{le:Young3} left hand side
of~\eqref{eq:Laplace_der_symmetrisierten_zweiten_kovarianten_Ableitung} is given by 
\begin{align*}
& \tilde \nabla^*\tilde \nabla \tilde R(x_1,x_2,x_3,x_4) & \stackrel{\eqref{eq:hirachy_2}}{=} &\frac{1}{4}\; \begin{array}{|c|c|}
\cline{1-2}
1 & 3 \\
\cline{1-2}
2 & 4 \\
\cline{1-2}
\end{array}\; \tilde \nabla^{2)}_{x_1,x_3}\tilde \ric(x_2,x_4)\\
& \stackrel{\eqref{eq:Ricci_tensor_of_the-associated_linear_two-jet_1}}{=}  \underbrace{- 2\, \begin{array}{|c|c|}
\cline{1-2}
1 & 3 \\
\cline{1-2}
2 & 4 \\
\cline{1-2}
\end{array}\; R*R(x_1,x_2,x_3,x_4)}_{\stackrel{\eqref{eq:eigenvalue}}{=} -24\, R*R(x_1,x_2,x_3,x_4)}
& + & \underbrace{4\, \begin{array}{|c|c|}
\cline{1-2}
1 & 3 \\
\cline{1-2}
2 & 4 \\
\cline{1-2}
\end{array}\;R_{x_1,x_2}\cdot\ric(x_3,x_4)}_{\stackrel{\eqref{eq:hilf4}}{=}0}\\
& + \underbrace{20 \, \begin{array}{|c|c|}
\cline{1-2}
1 & 3 \\
\cline{1-2}
2 & 4 \\
\cline{1-2}
\end{array}\ \nabla^{2)}_{x_1,x_3}\ric(x_{2},x_{4})}_{\stackrel{\eqref{eq:Weitzenböck_spezial}}{=}
80\,\nabla^*\nabla R(x_1,x_2,x_3,x_4) + 40\, R*R(x_1,x_2,x_3,x_4)}& & 
\end{align*}
The result follows.
\end{proof}

\section{The canonical embedding of Ricci-flat curvature tensors}
Given the algebraic curvature tensor $R$\,, we set
\begin{equation}\label{eq:associated_linear_two-jet_2}
\hat \nabla^{2)}_{x_5,x_6} \hat R(x_1,x_2,x_3,x_4) := \begin{array}{|c|c|c|c|}
\cline{1-4}
1 & 3 & 5 & 6 \\
\cline{1-4}
2 & 4 \\
\cline{1-2}
\end{array}\;\g{x_5}{x_6}R(x_1,x_2,x_3,x_4)\;.
\end{equation}
Then $\hat\nabla^{2)}\hat R\in\scrC_2$ and the linear map 
\begin{equation}\label{eq:iota}
\iota:\scrC_0 \to \scrC_2, R \mapsto \hat \nabla^{2)} \hat R
\end{equation}
is an embedding according to the arguments given in the proof of
Lemma~\ref{le:linear_jacobi_relation_alternativ}. In the following we suppose
further that $\ric = 0$\,. By definition, the second covariant derivative of
the Ricci-tensor and the rough Laplacian associated with $\hat\nabla^{2)}\hat R$ according
to~\eqref{eq:Ricci_tensor_of_a_linear_two-jet} are given by
\begin{align}\label{eq:Ricci_tensor_of_the_associated_linear_two-jet_2}
\hat \nabla^{2)}_{x_5,x_6}\hat\ric(x_2,x_4) = - \trace_{1,3}\,\begin{array}{|c|c|c|c|}
\cline{1-4}
1 & 3 & 5 & 6 \\
\cline{1-4}
2 & 4 \\
\cline{1-2}
\end{array}\;\g{x_5}{x_6}R(x_1,x_2,x_3,x_4)\;,\\
\label{eq:rough_Laplacian_of_the_associated_linear_two-jet_2}
\hat\nabla^*\hat\nabla\hat R(x_1,x_2,x_3,x_4) = - \trace_{5,6}\,\begin{array}{|c|c|c|c|}
\cline{1-4}
1 & 3 & 5 & 6 \\
\cline{1-4}
2 & 4 \\
\cline{1-2}
\end{array}\;\g{x_5}{x_6}R(x_1,x_2,x_3,x_4)\;.
\end{align}

\bigskip
\begin{lemma}
Let an algebraic curvature tensor $R$ with vanishing Ricci tensor be given. Then
\begin{equation}\label{eq:Young3}
\trace_{1,3}\;\begin{array}{|c|c|c|}
\cline{1-2}
1 & 3 \\
\cline{1-2}
2 & 4 \\
\cline{1-2}
\end{array}\; \g{x_1}{x_6}\, R(x_3,x_2,x_5,x_4) =  \begin{array}{ll}
3\,\big ( R(x_6,x_2,x_5,x_4) & + R(x_6,x_4,x_5,x_2) \big )\;.
\end{array}
\end{equation}
\end{lemma}
\begin{proof}
\begin{align*}
\begin{array}{|c|c|c|}
\cline{1-2}
1 & 3 \\
\cline{1-2}
2 & 4 \\
\cline{1-2}
\end{array}\; \g{x_1}{x_6}\, R(x_3,x_2,x_5,x_4) = 
\left \lbrace \begin{array}{ll}
+ \g{x_1}{x_6}\, R(x_3,x_2,x_5,x_4) & + \g{x_3}{x_6}\, R(x_1,x_2,x_5,x_4)\\
+ \g{x_1}{x_6}\, R(x_3,x_4,x_5,x_2) & + \g{x_3}{x_6}\, R(x_1,x_4,x_5,x_2)\\
- \g{x_2}{x_6}\, R(x_3,x_1,x_5,x_4) & - \g{x_3}{x_6}\, R(x_2,x_1,x_5,x_4)\\
- \g{x_2}{x_6}\, R(x_3,x_4,x_5,x_1) & - \g{x_3}{x_6}\, R(x_2,x_4,x_5,x_1)\\
- \g{x_1}{x_6}\, R(x_4,x_2,x_5,x_3) & - \g{x_4}{x_6}\, R(x_1,x_2,x_5,x_3)\\
- \g{x_1}{x_6}\, R(x_4,x_3,x_5,x_2) & - \g{x_4}{x_6}\, R(x_1,x_3,x_5,x_2)\\
+ \g{x_2}{x_6}\, R(x_4,x_1,x_5,x_3) & + \g{x_4}{x_6}\, R(x_2,x_1,x_5,x_3)\\
+ \g{x_2}{x_6}\, R(x_4,x_3,x_5,x_1) & + \g{x_4}{x_6}\, R(x_2,x_3,x_5,x_1)\\
\end{array} \right . \;.
\end{align*}
Hence,
\begin{align*}
&\trace_{1,3}\;\begin{array}{|c|c|c|}
\cline{1-2}
1 & 3 \\
\cline{1-2}
2 & 4 \\
\cline{1-2}
\end{array}\; \g{x_1}{x_6}\, R(x_3,x_2,x_4,x_5) = 
\left \lbrace \begin{array}{ll}
+ \g{e_i}{x_6}\, R(e_i,x_2,x_5,x_4) & + \g{e_i}{x_6}\, R(e_i,x_2,x_5,x_4)\\
+ \g{e_i}{x_6}\, R(e_i,x_4,x_5,x_2) & + \g{e_i}{x_6}\, R(e_i,x_4,x_5,x_2)\\
- \g{x_2}{x_6}\, R(e_i,e_i,x_5,x_4) & - \g{e_i}{x_6}\, R(x_2,e_i,x_5,x_4)\\
- \g{x_2}{x_6}\, R(e_i,x_4,x_5,e_i) & - \g{e_i}{x_6}\, R(x_2,x_4,x_5,e_i)\\
- \g{e_i}{x_6}\, R(x_4,x_2,x_5,e_i) & - \g{x_4}{x_6}\, R(e_i,x_2,x_5,e_i)\\
- \g{e_i}{x_6}\, R(x_4,e_i,x_5,x_2) & - \g{x_4}{x_6}\, R(e_i,e_i,x_5,x_2)\\
+ \g{x_2}{x_6}\, R(x_4,e_i,x_5,e_i) & + \g{x_4}{x_6}\, R(x_2,e_i,x_5,e_i)\\
+ \g{x_2}{x_6}\, R(x_4,e_i,x_5,e_i) & + \g{x_4}{x_6}\, R(x_2,e_i,x_5,e_i)\\
\end{array} \right .\\
&  = \left \lbrace \begin{array}{ll}
+ R(x_6,x_2,x_5,x_4) & + R(x_6,x_2,x_5,x_4)\\
+ R(x_6,x_4,x_5,x_2) & + R(x_6,x_4,x_5,x_2)\\
 & - R(x_2,x_6,x_5,x_4)\\
- \g{x_2}{x_6}\, \ric(x_4,x_5) & - R(x_2,x_4,x_5,x_6)\\
- R(x_4,x_2,x_5,x_6) & - \g{x_4}{x_6}\, \ric(x_2,x_5)\\
- R(x_4,x_6,x_5,x_2) & \\
- \g{x_2}{x_6}\, \ric(x_4,x_5) & - \g{x_4}{x_6}\, \ric(x_2,x_5)\\
- \g{x_2}{x_6}\, \ric(x_4,x_5) & - \g{x_4}{x_6}\, \ric(x_2,x_5)
\end{array} \right \rbrace = 3\, \left (\begin{array}{ll}
+ R(x_6,x_2,x_5,x_4) & + R(x_6,x_4,x_5,x_2)\\
- \g{x_2}{x_6}\, \ric(x_4,x_5) & - \g{x_4}{x_6}\, \ric(x_2,x_5)
\end{array} \right )\;.
\end{align*}
\end{proof}

\bigskip
\begin{corollary}
Let an algebraic curvature tensor $R$ with vanishing Ricci tensor be given. Then 
\begin{equation}\label{eq:Young4}
\trace_{1,3}\;\begin{array}{|c|c|c|}
\cline{1-3}
1 & 3 & 5\\
\cline{1-3}
2 & 4 & \multicolumn{1}{c}{\;\;\;} \\
\cline{1-2}
\end{array}\; \g{x_5}{x_6}\, R(x_1,x_2,x_3,x_4) = 6\,\left(\begin{array}{l}
R(x_6,x_2,x_5,x_4) + R(x_6,x_4,x_5,x_2)
\end{array}
\right )\;.
\end{equation}
\end{corollary}
\begin{proof}
\begin{align*}
& \begin{array}{|c|c|c|}
\cline{1-3}
1 & 3 & 5\\
\cline{1-3}
2 & 4 & \multicolumn{1}{c}{\;\;\;} \\
\cline{1-2}
\end{array}\;\g{x_5}{x_6}R(x_1,x_2,x_3,x_4)
 = &  \begin{array}{|c|c|}
\cline{1-2}
1 & 3 \\
\cline{1-2}
2 & 4 \\
\cline{1-2}
\end{array}\; \g{x_5}{x_6} R(x_1,x_2,x_3,x_4)\\
&  +\;\begin{array}{|c|c|}
\cline{1-2}
1 & 3 \\
\cline{1-2}
2 & 4 \\
\cline{1-2}
\end{array}\; \g{x_1}{x_6} R(x_3,x_2,x_5,x_4)
&+\;\begin{array}{|c|c|}
\cline{1-2}
1 & 3 \\
\cline{1-2}
2 & 4 \\
\cline{1-2}
\end{array}\; \g{x_3}{x_6}R(x_5,x_2,x_1,x_4)\;.
\end{align*}
Hence, 
\begin{align*}
& \trace_{1,3}\;\begin{array}{|c|c|c|}
\cline{1-3}
1 & 3 & 3\\
\cline{1-3}
2 & 4 & \multicolumn{1}{c}{\;\;\;} \\
\cline{1-2}
\end{array}\;\g{x_5}{x_6}R(x_1,x_2,x_3,x_4)
 = & \underbrace{\trace_{1,3}\;\begin{array}{|c|c|}
\cline{1-2}
1 & 3 \\
\cline{1-2}
2 & 4 \\
\cline{1-2}
\end{array}\; \g{x_5}{x_6} R(x_1,x_2,x_3,x_4)}_{\stackrel{\eqref{eq:eigenvalue}} = 0} \\
& +  \underbrace{\trace_{1,3}\;\begin{array}{|c|c|}
\cline{1-2}
1 & 3 \\
\cline{1-2}
2 & 4 \\
\cline{1-2}
\end{array}\; \g{x_1}{x_6} R(x_3,x_2,x_5,x_4)}_{\stackrel{\eqref{eq:Young3}}{=} \begin{array}{ll}
+ 3\, R(x_6,x_2,x_5,x_4) & + 3\, R(x_6,x_4,x_5,x_2)
\end{array}
}
& +\underbrace{\trace_{1,3}\;\begin{array}{|c|c|}
\cline{1-2}
1 & 3 \\
\cline{1-2}
2 & 4 \\
\cline{1-2}
\end{array}\; \overbrace{\g{x_3}{x_6}R(x_5,x_2,x_1,x_4)}^{= \g{x_3}{x_6}R(x_1,x_4,x_2,x_5)}}_{\stackrel{\eqref{eq:Young3}}= 
\begin{array}{ll}
+ 3\, R(x_6,x_2,x_5,x_4) & + 3\, R(x_6,x_4,x_5,x_2)
\end{array}
}\;.
\end{align*}
The result follows.
\end{proof}

\bigskip
\begin{lemma}
Let an algebraic curvature tensor $R$ with $\ric = 0$ be given. Then 
\begin{equation}\label{eq:Young5}
\trace_{1,3}\; \begin{array}{|c|c|}
\cline{1-2}
1 & 3 \\
\cline{1-2}
2 & 4 \\
\cline{1-2}
\end{array}\; \g{x_1}{x_3}\, R(x_5,x_2,x_6,x_4) = 
(2\,n - 4)\,\left (\begin{array}{l} 
 R(x_5,x_2,x_6,x_4) + R(x_5,x_4,x_6,x_2)
\end{array}\right )\;.
\end{equation}
\end{lemma}
\begin{proof}
\begin{align*}
& \begin{array}{|c|c|}
\cline{1-2}
1 & 3 \\
\cline{1-2}
2 & 4 \\
\cline{1-2}
\end{array}\; \g{x_1}{x_3}\, R(x_5,x_2,x_6,x_4) = \left \lbrace 
\begin{array}{ll} 
+ \g{x_1}{x_3}\, R(x_5,x_2,x_6,x_4) & + \g{x_3}{x_1}\, R(x_5,x_2,x_6,x_4)\\
+ \g{x_1}{x_3}\, R(x_5,x_4,x_6,x_2) & + \g{x_3}{x_1}\, R(x_5,x_4,x_6,x_2)\\
- \g{x_2}{x_3}\, R(x_5,x_1,x_6,x_4) & - \g{x_3}{x_2}\, R(x_5,x_1,x_6,x_4)\\
- \g{x_2}{x_3}\, R(x_5,x_4,x_6,x_1) & - \g{x_3}{x_2}\, R(x_5,x_4,x_6,x_1)\\
- \g{x_1}{x_4}\, R(x_5,x_2,x_6,x_3) & - \g{x_4}{x_1}\, R(x_5,x_2,x_6,x_3)\\
- \g{x_1}{x_4}\, R(x_5,x_3,x_6,x_2) & - \g{x_4}{x_1}\, R(x_5,x_3,x_6,x_2)\\
+ \g{x_2}{x_4}\, R(x_5,x_1,x_6,x_3) & + \g{x_4}{x_2}\, R(x_5,x_1,x_6,x_3)\\
+ \g{x_2}{x_4}\, R(x_5,x_3,x_6,x_1) & + \g{x_4}{x_2}\, R(x_5,x_3,x_6,x_1)
\end{array}
\right .\;.
\end{align*}
Hence,
\begin{align*}
& \trace_{1,3}\;\begin{array}{|c|c|}
\cline{1-2}
1 & 3 \\
\cline{1-2}
2 & 4 \\
\cline{1-2}
\end{array}\; \g{x_1}{x_3}\, R(x_5,x_2,x_4,x_6) & = &\left \lbrace 
\begin{array}{ll} 
+ \g{e_i}{e_i}\, R(x_5,x_2,x_6,x_4) & + \g{e_i}{e_i}\, R(x_5,x_2,x_6,x_4)\\
+ \g{e_i}{e_i}\, R(x_5,x_4,x_6,x_2) & + \g{e_i}{e_i}\, R(x_5,x_4,x_6,x_2)\\
- \g{x_2}{e_i}\, R(x_5,e_i,x_6,x_4) & - \g{e_i}{x_2}\, R(x_5,e_i,x_6,x_4)\\
- \g{x_2}{e_i}\, R(x_5,x_4,x_6,e_i) & - \g{e_i}{x_2}\, R(x_5,x_4,x_6,e_i)\\
- \g{e_i}{x_4}\, R(x_5,x_2,x_6,e_i) & - \g{x_4}{e_i}\, R(x_5,x_2,x_6,e_i)\\
- \g{e_i}{x_4}\, R(x_5,e_i,x_6,x_2) & - \g{x_4}{e_i}\, R(x_5,e_i,x_6,x_2)\\
+ \g{x_2}{x_4}\, R(x_5,e_i,x_6,e_i) & + \g{x_4}{x_2}\, R(x_5,e_i,x_6,e_i)\\
+ \g{x_2}{x_4}\, R(x_5,e_i,x_6,e_i) & + \g{x_4}{x_2}\, R(x_5,e_i,x_6,e_i)
\end{array}
\right .\\
& = 2\, \left ( 
\begin{array}{ll} 
+ n R(x_5,x_2,x_6,x_4) & + n R(x_5,x_4,x_6,x_2) \\
- R(x_5,x_2,x_6,x_4) & - R(x_5,x_4,x_6,x_4)\\
- R(x_5,x_2,x_6,x_4) & - R(x_5,x_4,x_6,x_2)\\
- \g{x_2}{x_4}\, \ric (x_5,x_6) & - \g{x_2}{x_4}\, \ric (x_5,x_6)
\end{array}
\right ) & = & \left \lbrace \begin{array}{l} 
+ (2\,n - 4)R(x_5,x_2,x_4,x_6)\\ + (2\,n - 4) R(x_5,x_4,x_2,x_6)\\ + 4
\g{x_2}{x_4}\, \ric (x_5,x_6)
\end{array} \right .\;.
\end{align*}
\end{proof}

\bigskip
\begin{corollary}\label{co:Ricci_tensor}
Let an algebraic curvature tensor $R$ with vanishing Ricci tensor be given. Then
the second covariant derivative of
the Ricci-tensor and the rough Laplacian accociated with~\eqref{eq:associated_linear_two-jet_2} are given by
\begin{enumerate}
\item \begin{equation}\label{eq:Ricci_tensor_of_iotaR}
\hat \nabla^{2)}_{x_5,x_6}\hat\ric(x_2,x_4) = - 4(n + 4)\big (R(x_5,x_2,x_6,x_4) + R(x_5,x_4,x_6,x_2)\big )\;,
\end{equation}
\item 
 \begin{equation}\label{eq:Rough_Laplacian_of_iotaR}
\hat\nabla^*\hat\nabla \hat R(x_1,x_2,x_3,x_4) = - 24(n + 4)\, R(x_1,x_2,x_3,x_4)\;.
\end{equation}
\end{enumerate}
\end{corollary}
\begin{proof}
Using the previous together with~\eqref{eq:symmetrizer},
\begin{align*}
& \frac{1}{2}\,\trace_{1,3}\;\begin{array}{|c|c|c|c|}
\cline{1-4}
1 & 3 & 5 & 6\\
\cline{1-4}
2 & 4 & \multicolumn{2}{c}{\;\;\;}\\
\cline{1-2}
\end{array}\; \g{x_5}{x_6}\, R(x_1,x_2,x_3,x_4) = & \underbrace{\trace_{1,3}\;\begin{array}{|c|c|c|}
\cline{1-3}
1 & 3 & 5\\
\cline{1-3}
2 & 4 & \multicolumn{1}{c}{\;\;\;}\\
\cline{1-2}
\end{array}\; \g{x_5}{x_6}\, R(x_1,x_2,x_3,x_4)}_{\stackrel{\eqref{eq:Young4}}= 
6\, R(x_6,x_2,x_5,x_4) + 6\, R(x_6,x_4,x_5,x_2)}\\
& +\underbrace{ \trace_{1,3}\;\begin{array}{|c|c|c|}
\cline{1-3}
1 & 3 & 6\\
\cline{1-3}
2 & 4 & \multicolumn{1}{c}{\;\;\;} \\
\cline{1-2}
\end{array}\; \g{x_5}{x_6}\, R(x_1,x_2,x_3,x_4)}_{\stackrel{\eqref{eq:Young4}}= 
6\, R(x_5,x_2,x_6,x_4) + 6\, R(x_5,x_4,x_6,x_2)}
& - \underbrace{\trace_{1,3}\;\begin{array}{|c|c|}
\cline{1-2}
1 & 3 \\
\cline{1-2}
2 & 4 \\
\cline{1-2}
\end{array}\; \g{x_5}{x_6}\, R(x_1,x_2,x_3,x_4)}_{= 0}\\
&  + \underbrace{\trace_{1,3}\;\begin{array}{|c|c|}
\cline{1-2}
1 & 3 \\
\cline{1-2}
2 & 4 \\
\cline{1-2}
\end{array}\; \g{x_1}{x_3}\, R(x_5,x_2,x_6,x_4)}_{\stackrel{\eqref{eq:Young5}}= 
(2\,n - 4)\, R(x_5,x_2,x_6,x_4) + (2\,n - 4)\, R(x_5,x_4,x_6,x_2)}\;.
\end{align*}
Eq.~\ref{eq:Ricci_tensor_of_iotaR} follows. Then~\eqref{eq:Rough_Laplacian_of_iotaR} follows from~\eqref{eq:hirachy_3}.
\end{proof}

\section{Proof of Theorem~\ref{th:main}}
On the analogy of Lemma~\ref{le:linear_jacobi_relation_alternativ}, it suffices to show the following:

\bigskip
\begin{singletheorem}[\bf Theorem 1']\label{th:main_alternativ}
The two-jet $(R,\nabla R,\nabla^{2)}R)$ of some curvature tensor is Einstein
if and only if $(R,\nabla R)$ is Einstein and 
\begin{equation}\label{eq:traceless_part_1}
\begin{array}{|c|c|c|c|}
\cline{1-4}
1 & 3 & 5 & 6 \\
\cline{1-4}
2 & 4 \\
\cline{1-2}
\end{array}\;\big ( \nabla^{2)}_{x_5,x_6}R(x_1,x_2,x_3,x_4) -
\frac{1}{n+4}\g{x_5}{x_6} (R*R)(x_1,x_2,x_3,x_4)\big )
\end{equation}
is totally trace-free.
\end{singletheorem}

\begin{proof}[Proof of Theorem 1']
In the one direction, suppose the Einstein condition holds. 
Clearly this implies that $R*\ric = 0$\,. Hence $R*R$ has vanishing Ricci tensor according to~\eqref{eq:Ricci_tensor_of_R*R}. Thus, on the one hand applying Corollary~\ref{co:Ricci_tensor} to the
curvature tensor $R*R$\,, we obtain that
\begin{align*}
&  - \frac{1}{n + 4} \trace_{1,3}\,\begin{array}{|c|c|c|c|}
\cline{1-4}
1 & 3 & 5 & 6 \\
\cline{1-4}
2 & 4 \\
\cline{1-2}
\end{array}\; \g{x_5}{x_6}\,R*R(x_1,x_2,x_3,x_4) \\ 
& \stackrel{\eqref{eq:Ricci_tensor_of_iotaR}}{=} - 4 \big (R*R(x_5,x_2,x_6,x_4) + R*R(x_5,x_4,x_6,x_2)\big )\;.
\end{align*}
On the other hand,  
\[
-\trace_{1,3}\, \begin{array}{|c|c|c|c|}
\cline{1-4}
1 & 3 & 5 & 6 \\
\cline{1-4}
2 & 4 \\
\cline{1-2}
\end{array}\; \nabla^{2)}_{x_5,x_6}R(x_1,x_2,x_3,x_4) = - 4\, \big (R*R(x_5,x_2,x_6,x_4) + R*R(x_5,x_4,x_6,x_2)\big )
\]
according to Corollary~\ref{co:Einstein}. Thus the trace
of~\eqref{eq:traceless_part_1} with respect to $\{x_1,x_3\}$ vanishes. Because of Lemma~\ref{le:hirachy} this
already implies that~\eqref{eq:traceless_part_1} is totally trace-free.
\qed

In the other direction, suppose that $(R,\nabla R)$ satisfies the Einstein condition and 
that~\eqref{eq:traceless_part_1} is totally trace-free. We aim to show that  $\nabla^{2)}\ric = 0$\,. For this, since the Ricci tensor is assumed to be a multiple of the identity,
\begin{align*}
& \nabla^{2)}_{x_1,x_2}\ric(y_1,y_2) - \nabla^{2)}_{x_2,x_1}\ric(y_1,y_2)\;,\\
& = R_{x_1,x_2}\cdot \ric(y_1,y_2) = 0\;.
\end{align*}
Thus we already know that
\begin{equation}
\label{eq:nabla_quadrat_in_Sym_Zwei_tensor_Sym_Zwei}
\nabla^{2)}\ric \in \Sym^2 V^* \otimes \Sym^2 V^*\;.
\end{equation}
Further, recall from the Littlewood-Richardson rules that there is an abstract $\sl(V;\C)$ decomposition
\begin{equation}\label{eq:Littlewood_Richardson_rule}
\Sym^2 V^* \otimes \Sym^2 V^* = \begin{array}{|c|c|c|c|}
\cline{1-4}
\quad &\quad  &\quad &\quad \\
\cline{1-4}
\end{array} \oplus \begin{array}{|c|c|c|}
\cline{1-3}
\quad &\quad  &\quad\\
\cline{1-3}
\quad\\
\cline{1-1}
\end{array} \oplus \begin{array}{|c|c|}
\cline{1-2}
\quad & \quad \\
\cline{1-2}
\quad & \quad \\
\cline{1-2}
\end{array}
\end{equation}
where each Young frame represents some irreducible component in $\bigotimes^4 V^*$\,.
We will show that each of the three components of  $\nabla^{2)}\ric$ vanishes:

First, we aim to show that 
\begin{equation}\label{eq:vanishes?_1}
\forall_{x_1,x_2,x_3,x_4\in V}\begin{array}{|c|c|c|c|}
\cline{1-4}
1 & 2 & 3 & 4\\
\cline{1-4}
\end{array}\; \nabla^{2)}_{x_1,x_3}\ric(x_2,x_4) = 0\;.
\end{equation} 
This means by the polarization formula that $\nabla^{2)}_{\xi,\xi}\ric(\xi,\xi) = 0$ for all $\xi\in V$\,.

For this, we calculate the trace of~\eqref{eq:traceless_part_1} with respect to $\{x_1,x_3\}$ and evaluate for $x_2 = x_4 = x_5 = x_6 = \xi$\,:
\begin{align*}
&\underbrace{\trace_{1,3}\,\begin{array}{|c|c|c|c|}
\cline{1-4}
1 & 3 & 5 & 6 \\
\cline{1-4}
2 & 4 \\
\cline{1-2}
\end{array}\; \nabla^{2)}_{x_5,x_6}R(x_1,x_2,x_3,x_4)|_{x_2 = x_4 = x_5 = x_ 6 = \xi}}_{\stackrel{\eqref{eq:Ricci_tensor_of_the-associated_linear_two-jet_1}}{=}
- 80\, \nabla^{2)}_{\xi,\xi}\ric(\xi,\xi)}\\
& \stackrel{!}{=} \underbrace{\frac{1}{n + 4}\,\trace_{1,3}\, \begin{array}{|c|c|c|c|}
\cline{1-4}
1 & 3 & 5 & 6 \\
\cline{1-4}
2 & 4 \\
\cline{1-2}
\end{array}\; \g{x_5}{x_6}R*R(x_1,x_2,x_3,x_4)|_{x_2 = x_4 = x_5 = x_ 6 = \xi}}_{\stackrel{\eqref{eq:Ricci_tensor_of_iotaR}}= 0}\;.
\end{align*}
We conclude that $\nabla^{2)}_{\xi,\xi}\ric(\xi,\xi) = 0$\,.\qed

Second, we will show that 
\begin{equation}\label{eq:vanishes?_2}
\forall_{x_1,x_2,x_3,x_4\in V}\,:\; \begin{array}{|c|c|c|}
\cline{1-3}
1 & 3 & 4\\
\cline{1-3}
2 \\
\cline{1-1}
\end{array}\; \nabla^{2)}_{x_1,x_2}\ric(x_3,x_4) = \begin{array}{|c|c|c|}
\cline{1-3}
1 & 2 & 4\\
\cline{1-3}
3 \\
\cline{1-1}
\end{array}\; \nabla^{2)}_{x_1,x_2}\ric(x_3,x_4) = 0\;.
\end{equation}
Thus it suffices to show that  $\nabla^{2)}_{\xi,x}\ric(\xi,\xi) = \nabla^{2)}_{\xi,\xi}\ric(x,\xi) = 0$ for all $x,\xi$\,:

we calculate the trace of~\eqref{eq:traceless_part_1} with respect to $\{x_1,x_3\}$ and evaluate for  $x_2 = x_4 = x_5 = \xi$\,, $x_6 = x$\,:
\begin{align*}
&\underbrace{\trace_{1,3}\,\begin{array}{|c|c|c|c|}
\cline{1-4}
1 & 3 & 5 & 6 \\
\cline{1-4}
2 & 4 \\
\cline{1-2}
\end{array}\; \nabla^{2)}_{x_5,x_6}R(x_1,x_2,x_3,x_4)|_{x_2 = x_4 = x_5 = \xi\,,\ x_6=x}}_{\stackrel{\eqref{eq:Ricci_tensor_of_the-associated_linear_two-jet_1},\ R\cdot \ric = 0}{=} 80\, \nabla^{2)}_{\xi,x}\ric(\xi,\xi)}\\
& \stackrel{!}{=}  \underbrace{\frac{1}{n + 4}\,\trace_{1,3}\, \begin{array}{|c|c|c|c|}
\cline{1-4}
1 & 3 & 5 & 6 \\
\cline{1-4}
2 & 4 \\
\cline{1-2}
\end{array}\; \g{x_5}{x_6}R*R(x_1,x_2,x_3,x_4)|_{x_2 = x, x_4 = x_5 = x_6 = \xi}}_{\stackrel{\eqref{eq:Ricci_tensor_of_iotaR}}= 0}\;.
\end{align*}
We conclude that $\nabla^{2)}_{\xi,x}\ric(\xi,\xi) = 0$\,. The conclusion that $\nabla^{2)}_{\xi,\xi}\ric(x,\xi) = 0$ is derived in the same way.\qed

Last we have to show that 
\begin{equation}\label{eq:vanishes?_3}
\forall_{x_1,x_2,x_3,x_4\in V}\,:\; \begin{array}{|c|c|}
\cline{1-2}
1 & 3 \\
\cline{1-2}
2 & 4 \\
\cline{1-2}
\end{array}\; \nabla^{2)}_{x_1,x_3}\ric(x_2,x_4) = 0\;.
\end{equation}

For this, we calculate the trace of~\eqref{eq:traceless_part_1} with respect to $\{x_5,x_6\}$ to see that
\begin{align*}
&\underbrace{\trace_{5,6}\,\begin{array}{|c|c|c|c|}
\cline{1-4}
1 & 3 & 5 & 6 \\
\cline{1-4}
2 & 4 \\
\cline{1-2}
\end{array}\; \nabla^{2)}_{x_5,x_6}R(x_1,x_2,x_3,x_4)}_{\stackrel{\eqref{eq:Laplace_der_symmetrisierten_zweiten_kovarianten_Ableitung}}{=} -80\, \nabla^*\nabla R(x_1,x_2,x_3,x_4) - 16\, R*R(x_1,x_2,x_3,x_4)} \stackrel{!}{=} \underbrace{\frac{1}{n + 4}\trace_{5,6}\, \begin{array}{|c|c|c|c|}
\cline{1-4}
1 & 3 & 5 & 6 \\
\cline{1-4}
2 & 4 \\
\cline{1-2}
\end{array}\; \g{x_5}{x_6}R*R(x_1,x_2,x_3,x_4)}_{\stackrel{\eqref{eq:Rough_Laplacian_of_iotaR}}= 24\, R*R(x_1,x_2,x_3,x_4)}\;.
\end{align*}
We conclude that $\nabla^*\nabla R +\frac{1}{2} R*R = 0$\,. Using the special case of the Weitzenböck formula~\eqref{eq:Weitzenböck_spezial}, we see that~\eqref{eq:vanishes?_2} vanishes.
\end{proof}

\section{Proof of Corollary~\ref{co:main}}
Let $V$ be a pseudo-euclidean space with $n := \dim(V)$\,. Recall from Sec.~\ref{se:reformulation} that $\scrC_k\otimes \C$  is an
irreducible representation of $\sl(V,\C)$ of highest weight $(k + 2)\,L_1 +
2\, L_2$\,. The subset $[k+2,2]\subset \scrC_k$ given by the totally trace-free tensors is
a representation of $\so(V)$\,. According to~\cite[Thm.~19.22]{FH}, the complexification $[k+2,2]\otimes\C$  is
\begin{itemize}
\item an irreducible representation  of $\so(V,\C)$ with highest weight $(k + 2)\, L_1 + 2\, L_2$ (if $n\geq 5$), 
\item the sum of two irreducible representations of $\so(V,\C)$ with highest weights $(k + 2)\, L_1 \pm 2\, L_2$ (if $n = 4$),  
\item or $[k + 2,2] = \{0\}$ (if $n\leq 3$)\;.
\end{itemize}
This is called Weyls construction for orthogonal groups. Now we can give the proof of Corollary~\ref{co:main}:

\begin{proof}
Suppose that~\eqref{eq:Relation_of_order_two} holds. By means of Lemma~\ref{le:linear_jacobi_relation_alternativ},
\begin{equation}\label{eq:traceless_part_2}
\trace_{1,3}\; \begin{array}{|c|c|c|c|}
\cline{1-4}
1 & 3 & 5 & 6 \\
\cline{1-4}
2 & 4 \\
\cline{1-2}
\end{array}\;\big ( \nabla^{2)} R_{x_5,x_6}(x_1,x_2,x_3,x_4) - c\,\g{x_5}{x_6}R(x_1,x_2,x_3,x_4) \big ) = 0\;.
\end{equation} 
Using additionally~\eqref{eq:traceless_part_1} and substituting $R*R = - 2\nabla^*\nabla R$ according to~\eqref{eq:Weitzenböck_spezial_fuer_Einstein}, we thus see that
\begin{equation}\label{eq:traceless_part_3}
\trace_{1,3}\; \begin{array}{|c|c|c|c|}
\cline{1-4}
1 & 3 & 5 & 6 \\
\cline{1-4}
2 & 4 \\
\cline{1-2}
\end{array}\;\big (\frac{2}{n + 4}\g{x_5}{x_6}\nabla^*\nabla R_{x_5,x_6}(x_1,x_2,x_3,x_4) + c\, \g{x_5}{x_6}R(x_1,x_2,x_3,x_4) \big ) = 0\;.
\end{equation} 
Further, recall from~\eqref{eq:decomposition_of_algebraic_curvature_tensors} that on the one hand there is a decomposition into $\SO(V)$-modules
\[
\scrC_0 = [0] \oplus [2] \oplus [2,2]\;.
\]
Let $[4,2]\subset\scrC_2$ denote the subset of totally trace-free tensors. Since the embedding $\iota$ defined in~\eqref{eq:iota} is clearly $\SO(V)$-equivariant, Schurs Lemma in combination with~\eqref{eq:decomposition_of_algebraic_curvature_tensors} implies that $[4,2]\cap \iota(\scrC_0)
=\{0\}$\,. Hence $\hat \nabla^{2)} \hat R$ is completely determined by its traces and then by $\trace_{1,3}$ as was shown in Lemma~\ref{le:hirachy}.
We conclude from~\eqref{eq:traceless_part_3} that
\begin{equation}
\frac{2}{n + 4}\nabla^*\nabla R_{x_5,x_6}(x_1,x_2,x_3,x_4) + c\, R(x_1,x_2,x_3,x_4) = 0\;.
\end{equation}
Thus $\nabla^*\nabla R = -\frac{(n + 4)c}{2} R$\,, which finishes the proof of~\eqref{eq:main}.
\end{proof}

\section{Concluding remarks}
Let $(V,\gg)$ be a pseudo-euclidean vector space. It is known from the theory
of partial differential equations that an algebraic $k$-jet is Einstein if and
only if it is actually the $k$-jet of the curvature tensor of some Einstein metric defined in a
neighborhood of the origin of $V$ (for $k=0$ see~\cite{Ga}.) It follows (in complete analogy to the Jet Isomorphism Theorem) that the space of Einstein $k$-jets is an affine vector bundle over the space of Einstein $k-1$-jets with direction space $[k+2,2]$\,, the totally traceless part of $\scrC_{k}$\,. 

However, to the authors best knowledge, this existence result does not give a hint how to construct the Einstein metric explicitly from the given $k$-jet. 
Now Theorem~\ref{th:main_alternativ} points exactly into that direction. Namely, it tells us how to extend a given Einstein one-jet $(R,\nabla R)$ to an Einstein two-jet in an explicit way:

we may extend $(R,\nabla R)$ to some two-jet $(R,\nabla R,\nabla^{2)}R)$\,. In
fact, 
\[
\Id - \frac{1}{3}\scrR - \frac{1}{6} \scrR^{1)}:\xi\mapsto \Id -
\frac{1}{3}R(\,\cdot\,,\xi,\xi) - \frac{1}{6} \nabla_\xi
R(\,\cdot\,,\xi,\xi)\in \End_+(V)
\]
is a polynomial of degree three with values in the symmetric endomorphisms of
$V$\,. Using the inner product $\gg$ in order to identify symmetric
endomorphisms with symmetric bilinear forms, we obtain a metric defined in
a neighbourhood of the origin which has the prescribed two-jet $(R,\nabla R)$\,. 
Next, we remove the traces of~\eqref{eq:traceless_part_1} via adding a suitable $\hat\nabla^{2)} \hat R\in\scrC_2$ (which is the analogue of finding the Weyl part of an algebraic curvature tensor.)
Then $(R,\nabla R,\nabla^{2)}R + \frac{1}{80}\hat\nabla^{2)} \hat R)$ is an Einstein two-jet.

It seems reasonable that similar ideas also work for higher $k$-jets (i.e. $k\geq 3)$\,.

\end{document}